\documentclass[11pt]{article}
\usepackage{amsmath}
\usepackage{amsthm}
\usepackage{amsfonts}
\usepackage{dsfont}
\usepackage{enumerate}
\usepackage{graphicx}
\usepackage{latexsym}
\usepackage{url}
\usepackage{bbm}
\usepackage{mathtools}
\mathtoolsset{showonlyrefs}

 \usepackage{todonotes}

\newcommand{\R}{\mathbb{R}}

\newcommand{\E}{\mathbb{E}}
\newcommand{\PP}{\mathbb{P}}
\newcommand{\N}{\mathbb{N}}

\newcommand{\Q}{\mathbb{Q}}
\newcommand{\F}{\mathbb{F}}
\newcommand{\1}{\mathds{1}}
\renewcommand{\P}{\mathbb P}

\newcommand{\f}{\mathcal F}

\newcommand{\lll}{\mathcal L}

\newcommand{\toop}{\stackrel{\PP}{\longrightarrow}}

\newcommand{\eqschw}{\stackrel{d}{=}}
\newcommand{\stab}{\stackrel{\mathcal{L}-s}{\longrightarrow}}
\newcommand{\tols}{\stackrel{\lll- s}{\longrightarrow}~}
\newcommand{\tol}{\stackrel{\lll}{\longrightarrow}}

\newcommand{\ucp}{\xrightarrow{\mbox{\tiny u.c.p.}}}

\newcommand{\toas}{\stackrel{\mbox{\tiny a.s.}}{\longrightarrow}}

\newcommand{\bee}{\begin{equation}}
\newcommand{\eee}{\end{equation}}
\newcommand{\beea}{\begin{array}}
\newcommand{\eeea}{\end{array}}

\renewcommand{\theequation}{\arabic{section}.\arabic{equation}}

\theoremstyle{plain}
\newtheorem{prop}{Proposition}[section]
\newtheorem{cor}[prop]{Corollary}
\newtheorem{theo}[prop]{Theorem}
\newtheorem{lem}[prop]{Lemma}

\theoremstyle{definition}

\newtheorem{rem}[prop]{Remark}

\newcommand{\din}{\Delta_{i,k}^n}

\newcommand{\suonetn}{\sum_{i=k}^{[nt]}}
\newcommand{\lj}{{\scriptscriptstyle >a}}
\newcommand{\sj}{{\scriptscriptstyle \leq a}}
\newcommand{\kappal}{{\scriptscriptstyle \geq \kappa}}
\newcommand{\kappas}{{\scriptscriptstyle <\kappa}}
\newcommand{\llj}{{\scriptscriptstyle >a,> -\delta}}
\newcommand{\slj}{{\scriptscriptstyle >a,\leq -\delta}}
\newcommand{\ldl}[2]{\ensuremath{{\bf L}^{#1}( d #2)}}
\newcommand{\ldlnr}[2]{\ensuremath{{\bf L}_{\text{nr}}^{#1}( d #2)}}
\newcommand{\eps}{\varepsilon}

\newcommand{\skorcon}{ \xrightarrow{\tiny \lll_{M_1}-s}}

\newcommand{\diff}{\; d}
\newcommand{\wt}{\widetilde}

\newcommand{\eqspace}{\hspace{4em}}

\begin{document}

\title{On  limit theory for L\'evy semi-stationary  processes}
\author{Andreas Basse-O'Connor\thanks{Department
of Mathematics, Aarhus University, 
E-mail: basse@math.au.dk.}  \and
Claudio Heinrich\thanks{Department
of Mathematics, Aarhus University, 
E-mail: claudio.heinrich@math.au.dk}  \and
Mark Podolskij\thanks{Department
of Mathematics, Aarhus University,
E-mail: mpodolskij@math.au.dk.}}

\maketitle

\begin{abstract}
In this paper we present some limit theorems for power variation   
of  L\'evy semi-stationary  processes in the setting of infill asymptotics. 
L\'evy semi-stationary  processes, which are a one-dimensional analogue 
of ambit fields, are moving average type processes with a multiplicative 
random component, which is usually referred to as volatility or intermittency.
From the mathematical point of view this work extends the asymptotic theory
investigated in \cite{BLP}, where  the authors derived the limit theory 
for $k$th order increments of  
stationary increments L\'evy driven moving averages. The asymptotic results turn out to
heavily depend 
on the interplay between the given order of the increments,
the considered power $p>0$, the Blumenthal--Getoor index $\beta \in (0,2)$ of the driving pure jump 
L\'evy process $L$ and the behaviour of the kernel function
$g$ at $0$ determined by the power $\alpha$. In this paper we will study the 
first order asymptotic theory for L\'evy semi-stationary  processes with a random volatility/intermittency
component and present some statistical applications of the probabilistic results.

\ \

\bigskip


\end{abstract}

\section{Introduction and main results} \label{Intro}
\setcounter{equation}{0}
\renewcommand{\theequation}{\thesection.\arabic{equation}}

Over the last ten years there has been a growing interest in the theory of ambit fields.
Ambit fields is a class of spatio-temporal stochastic processes that has been
originally introduced by Barndorff-Nielsen and Schmiegel in a series of papers \cite{BS07, BS08a, BS09} in
the context of turbulence modelling, but which has found manifold applications in 
mathematical finance and biology among other sciences; see e.g. \cite{BBC, BJJS07}.

Ambit processes describe the dynamics in a stochastically developing
field, for instance a turbulent wind field, along curves embedded in
such a field.   A key characteristic of the modelling framework is that
beyond the most basic kind of random noise it also specifically
incorporates additional, often drastically changing, inputs referred
to as \textit{volatility} or \textit{intermittency}.  In terms of mathematical formulae an ambit field is specified via
\begin{align} \label{ambfield}
  X_{t}(x)=\mu +\int_{A_{t}(x)}g(t,s,x,\xi )\sigma_{s}(\xi )\,L(d s, d
  \xi )+\int_{D_{t}(x)}q(t,s,x,\xi )a_{s}(\xi )\,d s\,d \xi,
\end{align}
where $t$ denotes time while $x$\ gives the position in
space. Further, $A_{t}(x)$ and $D_{t}(x)$ are Borel measurable subsets of $\R\times \R^d$, $g$ and $q$
are deterministic weight functions, $\sigma$ represents the volatility or
intermittency field, $a$ is a drift field and $L$ denotes an independently scattered 
infinitely divisible random measure on $\R\times \R^d$ (see e.g. \cite{RajRos} for details).  In the literature, the sets  $A_{t}(x)$ and $D_{t}(x)$ are usually referred to as \textit{ambit sets}. 
In the framework 
of turbulence modelling  the stochastic field $(X_{t}(x))_{t\geq 0,\, x\in \R^3}$ describes the velocity
of a turbulent flow at time $t$ and position $x$, while the ambit sets $A_{t}(x), D_{t}(x)$ are typically  
bounded. 

In this paper we will consider a purely temporal analogue of ambit fields (without drift)
$(X_t)_{t \in \R}$, defined on a filtered probability space $(\Omega, \mathcal F, (\mathcal F_t)_{t\in\R},  \mathbb P)$, which is given 
as  
\begin{align} \label{lss}
X_t= \int_{-\infty}^t \big\{g(t-s)-g_0(-s)\big\} \sigma_{s-}\, dL_s,
\end{align} 
and is usually referred to as a \textit{L\'evy semi-stationary process}. Here $L=(L_t)_{t\in \R}$ is a symmetric L\'evy process on $\R$ with respect to $(\mathcal F_t)_{t\in\R}$ with $L_0=0$ and 
without a Gaussian component. That is, for all $u\in\R,$ the process $(L_{t+u}-L_u)_{t\geq 0}$ is a symmetric L\'evy process on $\R_+$ with respect to $(\mathcal F_{t+u})_{t\geq 0}$. Moreover, $(\sigma_t)_{t\in \R}$
is a c\`adl\`ag process adapted to $(\mathcal F_t)_{t\in\R}$, and $g$ and $g_0$ are deterministic functions from $\R$ into $\R$  vanishing on $(-\infty,0)$. The name L\'evy semi-stationary process refers to the fact that the process 
$(X_t)_{t \in \R}$ is stationary whenever $g_0=0$ and $(\sigma_t)_{t\in \R}$ is stationary and  independent of $(L_t)_{t\in \R}$. It is assumed throughout this paper that $g,g_0, \sigma$ and $L$ are such that the process $(X_t)$ is well-defined, which will in particular be satisfied under the conditions stated in   Remark~\ref{Xintegrable} below. We are interested in the asymptotic behaviour of 
power variation of the process $X$. More precisely, let us consider the $k$th order increments 
$\Delta_{i,k}^{n} X$ of $X$, $k\in \N$, that are defined by 
\begin{align} \label{filter}
\Delta_{i,k}^{n} X:= \sum_{j=0}^k (-1)^j \binom{k}{j} X_{(i-j)/n}, \qquad i\geq k.
\end{align}
For instance, we have that $\Delta_{i,1}^n X = X_{\frac{i}{n}}-X_{\frac{i-1}{n}}$ and $\Delta_{i,2}^n X = X_{\frac{i}{n}}-
2X_{\frac{i-1}{n}}+X_{\frac{i-2}{n}}$. The main functional of interest  
is the power variation processes computed on the basis of $k$th order increments:
\begin{align} \label{vn}
V(p;k)^n_t := \sum_{i=k}^{[nt]} |\Delta_{i,k}^{n} X|^p, \qquad p>0.
\end{align}
At this stage we remark that power variation of stochastic processes has been a very active research 
area in the last decade. We refer e.g.\  to \cite{BGJPS,J,JP,PV} for limit theory for power variations of It\^o semimartingales, to \cite{BNCP09,BNCPW09,Coeu,gl89,nr09} 
for the asymptotic results in the framework of fractional Brownian motion and related processes, and to \cite{Rosenblatt, Rosenblatt-1} for investigations of power variation of the Rosenblatt process. More specifically, power variation of Brownian semi-stationary processes, which is the model \eqref{lss}
driven by a Brownian motion,  has been studied in \cite{BCP11,BCP13,GP}. Under proper
normalisation the authors have shown convergence in probability for the statistic $V(p;k)^n_t$ and proved its asymptotic mixed normality. 

However, when the driving motion in \eqref{lss} is a pure jump L\'evy process, the asymptotic theory is very different from the Brownian case. To see this, let us recall the first order asymptotic results investigated in  \cite{BLP}, who studied power variation of  a class of L\'evy semi-stationary processes 
with $\sigma=1$ and $t=1$.  Throughout the paper 
we will need the notion of \textit{Blumenthal--Getoor index} of $L$, which is defined via
\begin{align} \label{def-B-G}
\beta:=\inf\Big\{r\geq 0: \int_{-1}^1 |x|^r\,\nu(dx)<\infty\Big\}\in [0,2],
\end{align} 
where $\nu$ denotes the L\'evy measure of $L$. It is well-known that $\sum_{s \in [0,1]}
|\Delta L_s|^p$ is finite when $p>\beta$, while it is infinite for $p<\beta$.  Here $\Delta L_s=L_s-L_{s-}$ where $L_{s-}=\lim_{u\uparrow s,\,u<s} L_u$. The paper  \cite{BLP} imposes
the following set of assumptions on $g$, $g_0$ and $\nu$:
   
\noindent \textbf{Assumption~(A):}
The function $g\!:\R\to\R$ satisfies 
\begin{align}\label{kshs}
g(t)\sim c_0 t^\alpha\qquad \text{as } t\downarrow 0\quad \text{for some }\alpha>0\text{ and }c_0\neq  0, 
\end{align}
where $g(t)\sim f(t)$ as $t\downarrow 0$ means that $\lim_{t\downarrow 0}g(t)/f(t)= 1$. 
For some $\theta\in (0,2]$, $\limsup_{t\to \infty} \nu(x\!:|x|\geq t) t^{\theta}<\infty$ and $g - g_0$ is a bounded function in $L^{\theta}(\R_+)$.
 Furthermore,  $g$ is $k$-times continuously differentiable on $(0,\infty)$ and there  exists a $\delta>0$ such that  $|g^{(k)}(t)|\leq C t^{\alpha-k}$ for all $t\in (0,\delta)$, and such that both $|g'|$ and $|g^{(k)}|$ are in $L^\theta((\delta,\infty))$ and  are decreasing on $(\delta,\infty)$.

\noindent \textbf{Assumption~(A-log):}
In addition to  (A) suppose that $\int_\delta^\infty |g^{(k)}(s)|^\theta \log(1/|g^{(k)}(s)|)\,ds<\infty$. 

Assumption~(A) ensures, in particular,  that the process $X$ with $\sigma=1$ is well-defined, cf.\ \cite{BLP}.  
When $L$ is a $\beta$-stable L\'evy process, we can and will always choose $\theta = \beta$ in assumption (A).
Before we proceed with the main statement of \cite{BLP}, we need some more notation. 
Let $h_k\!:\R\to\R$ be given by 
\begin{align}\label{def-h-13}
h_k(x)&=  \sum_{j=0}^k (-1)^j \binom{k}{j} (x-j)_{+}^{\alpha},\qquad x\in \R,
\end{align} 
where  $y_+=\max\{y,0\}$ for all $y\in \R$. 
Let $\F=(\f_t)_{t\geq 0}$ and $(T_m)_{m\geq 1}$  be a sequence of $\F$-stopping times  that exhausts the jumps of $(L_t)_{t\geq 0}$. That is,  $ \{T_m(\omega):m\geq 1\}\cap \R_+ = \{t\geq 0: \Delta L_t(\omega)\neq 0\}$ and  $T_m(\omega)\neq T_n(\omega)$ for all $m\neq n$ with $T_m(\omega)<\infty$. Let     $(U_m)_{m\geq 1}$ be independent and uniform $[0,1]$-distributed random variables, defined on an extension $(\Omega', \mathcal F', \mathbb P')$ of the original probability space,
which are independent of $\f$. 

The following  first order limit theory for the power variation $V(p;k)_1^n$ has been proved 
in \cite{BLP} for $\sigma \equiv 1$. We refer to \cite{Aldous,ren}
for the definition of $\mathcal F$-stable convergence in law which will be denoted $\stab$. Moreover,   $\toop$ will denote convergence in probability.  
   
\begin{theo}[First order asymptotics \cite{BLP}] \label{maintheorem}
Suppose that $X=(X_t)_{t\geq 0}$ is a stochastic process defined by \eqref{lss} with $\sigma \equiv 1$. Suppose (A) is satisfied and assume that the Blumenthal--Getoor index satisfies $\beta<2$. 
Set $V(p;k)^n:=V(p;k)^n_1$.
We have  the following three cases:
 
\begin{itemize}
\item[(i)] \label{mt-case-1} Suppose that (A-log) holds if  $\theta=1$. If $\alpha <k-1/p$ and $p>\beta$ then  the $\f$-stable convergence holds as $n\to\infty$ 
\begin{equation} \label{part1}
n^{\alpha p}V(p;k)^n \stab |c_0|^p\!\!\!\!\! \!\sum_{m:\, T_m\in [0,1]} |\Delta L_{T_m}|^p V_m\quad\text{where}\quad 
V_m=\sum_{l=0}^{\infty} |h_k(l+U_m)|^p.
\end{equation} 
\item[(ii)] \label{mt-case-2} Suppose that  $L$ is a symmetric $\beta$-stable L\'evy process with scale parameter $\gamma>0$. If  $\alpha <k-1/\beta$ and   $p<\beta$ then it holds
\begin{align} \label{part2}
n^{-1+p(\alpha + 1/\beta)}V(p;k)^n \toop m_p
\end{align}
where  $m_p=|c_0|^p \gamma^p (\int_\R |h_k(x)|^\beta\,dx)^{p/\beta}\E[|Z|^p]$ and $Z$ is a symmetric $\beta$-stable random variable with scale parameter $1$. 

\item[(iii)] \label{mt-case-3} Suppose that $p\geq 1$. If  $p= \theta$ suppose in addition that (A-log) holds. For all    $\alpha>k-1/( \beta\vee p)$ 
we have that 
\begin{equation} \label{part3}
  n^{-1+pk}V(p;k)^n \toop \int_0^1 |F_u|^p\, du 
 \end{equation}
where $(F_u)_{u\in \R}$ is a measurable process satisfying
\begin{equation}
F_u=  \int_{-\infty}^u g^{(k)}(u-s) \,dL_s\quad \text{a.s.\ for all }u\in \R \quad \text{and}\quad \int_0^1 |F_u|^{p}\,du<\infty\quad \text{a.s.}
\end{equation} 
\end{itemize}
\end{theo}

The aim of this work is twofold. Firstly, we extend Theorem \ref{maintheorem} to L\'evy semi-stationary processes with a 
non-trivial volatility process  $\sigma$. Such extensions are important in applications, say in the framework of turbulence, 
since the volatility $\sigma$ are often of key importance.  Secondly,  we show that the convergence in all the three cases are functional in  the Skorokhod topology or in the uniform norm, see Theorem~\ref{theorem-sigma} below.  
As we will see later, first order asymptotic theory for L\'evy semi-stationary processes can be used to draw inference on the parameters $\alpha$, $\beta$ and 
on certain volatility functionals in the context of high frequency observations, see Section~\ref{sec3}.  Furthermore, 
this type of limit theory is an intermediate step towards asymptotic results for general ambit fields of the form    \eqref{ambfield}. We remark that, in contrast to the Brownian setting, the extension  of
Theorem \ref{maintheorem} to L\'evy semi-stationary processes is a more complex issue. This is due 
to the fact that it is harder to estimate various norms of $X$ and related processes when the driving process $L$ is a L\'evy process. Our estimates on  $X$ rely heavily on  decoupling techniques and isometries for stochastic integral mappings presented in the book of Kwapi\'en and Woyczy\'nski~\cite{KW},  see Section~\ref{sec2} for more details.  
To state our main result we introduce the following assumptions. 

\bigskip
\noindent
\textbf{Assumptions:} Throughout this paper we suppose  that (A) holds  and let $(H_s)_{s\in \R}$ denote the stochastic process  $H_s= g^{(k)}(-s)\sigma_s\mathds 1_{(-\infty,-\delta]}(s)$, $s\in \R$, where the constant $\delta$ is defined in assumption (A).  We now state two  additional  integrability assumptions on process $H$ to be used in  Theorem~\ref{theorem-sigma} below.

\medskip
\noindent
\textbf{Assumption~(B1):}
Suppose  there exists $\rho>0$ with $\rho\leq 1\wedge \theta$ and $\beta'>0$ with 
$\beta'>\beta$ and $\beta'\geq  p$ such that 
\begin{equation}\label{eq:A-b1}
\E\Big[ \Big( \int_{\R} \big(|H_s|^\rho \vee |H_s|^{\beta'} \big) \,ds          \Big)^{1\vee \frac{p}{2}}\Big]<\infty. 
\end{equation} 
For $\theta=1$ suppose in addition that we may choose $\rho<1$ in \eqref{eq:A-b1}.  

\medskip
\noindent
\textbf{Assumption~(B2):} Suppose that 
\begin{equation}
 \E\Big[ \Big(\int_{\R} |H_s|^{\beta}  \,ds  \Big)        \Big]<\infty. 
\end{equation}

Assumption~(B2) will only be used  in case $L$ is a symmetric $\beta$-stable L\'evy process, where 
we have $\theta=\beta$. In this case we note that (B1) is a stronger assumption than (B2).  
Let in the following $\mathbb D(\R_+;\R)$ denote the Skorokhod space of c\`adl\`ag functions from $\R_+$ into $ \R,$ equipped with the Skorokhod $M_1$-topology. For a detailed introduction and basic properties of this space we refer to \cite[Chapter 11.5]{Whitt}. We denote by $\skorcon$ the $\mathcal F$-stable convergence of c\`adl\`ag stochastic processes, regarded as random variables taking values in the Polish space $\mathbb D(\R_+;\R)$. By $\ucp$ we denote uniform convergence on compact sets in probability. That is, $(Y_t^n)_{t\geq 0}\ucp (Y_t)_{t\geq 0}$ as $n\to\infty$ means that $$\P(\sup_{t\in [0,N]}|Y^n_t-Y_t|>\eps)\to 0$$ for all $N\in\N$ and all $\eps>0.$ Below, let  $(T_m)_{m\geq 1}$ and $(U_m)_{m\geq 1}$ be defined as before Theorem~\ref{maintheorem} and the constant  $m_p$ be defined  as in  Theorem~\ref{maintheorem}(ii).

The following extension of Theorem~\ref{maintheorem}, to include a  non-trivial $\sigma$ process and functional convergence,  is the main result of this paper.

\begin{theo}\label{theorem-sigma}
 Let $X=(X_t)_{t\geq 0}$ be a stochastic process defined by \eqref{lss}. Assume that the Blumenthal--Getoor index satisfies $\beta<2$.
 \begin{itemize}
 \item [(i)]   Suppose that  (B1) holds and  that 
$\alpha <k-1/p$, $p>\beta$ and $p \geq 1$.    
Then, as $n\to\infty,$ the functional $\f$-stable convergence holds 
\begin{equation} \label{mainpart1}
n^{\alpha p}V(p;k)^n_t \skorcon |c_0|^p\!\!\!\!\! \!\sum_{m:\, T_m\in [0,t]} |\Delta L_{T_m}
\sigma_{T_m-}|^p V_m\quad\text{where}\quad 
V_m=\sum_{l=0}^{\infty} |h_k(l+U_m)|^p.
\end{equation} 

\item [(ii)]  Suppose  that $L$ is a symmetric $\beta$-stable L\'evy process with $\beta \in (0,2)$
and scale parameter $\gamma>0$. Suppose that (B2) holds and that $\alpha <k-1/\beta$ and $p<\beta$.  
Then  as $n\to \infty$
\begin{align} \label{mainpart2}
n^{-1+p(\alpha + 1/\beta)}V(p;k)^n_t \ucp m_p \int_0^t |\sigma_s|^p ds.
\end{align}
\item [(iii)]   Suppose that (B1) holds,  $\theta>1$,  $\alpha>k-1/( \beta\vee p)$ and $p\geq 1$. Then 
as $n\to \infty$ 
\begin{equation} \label{mainpart3}
  n^{-1+pk}V(p;k)^n_t \ucp \int_0^t |F_u|^p\, du ,
 \end{equation}
where $(F_u)_{u\in \R}$ is a measurable process satisfying
\begin{equation}
F_u=  \int_{-\infty}^u g^{(k)}(u-s) \sigma_{s-} \,dL_s\quad \text{a.s.\ for all }u\in \R \quad \text{and}\quad \int_0^t |F_u|^{p}\,du<\infty\quad \text{a.s.}
\end{equation} 
 \end{itemize}
\end{theo}

\begin{rem}
Under the integrability assumption (B1),  Theorem~\ref{theorem-sigma} covers all possible choices of $\alpha>0,\beta\in [0,2)$ and  $p\geq 1$
except the critical cases where $p= \beta$,   $\alpha= k-1/p$ or $\alpha= k-1/\beta$. The two critical cases
$\alpha= k-1/p$, $p>\beta$ and $\alpha= k-1/\beta$, $p<\beta$ have been discussed in \cite{BP15} in the case $\sigma \equiv 1$.  
\end{rem}

\begin{rem}
In his original work \cite{Skor}, Skorokhod introduced 4 different topologies on $\mathbb D(\R_+;\R)$, commonly referred to as  $J_1,J_2, M_1$ and $M_2$ topology, $J_1$ being by far the most popular one.
It can be shown that the functional stable convergence in Theorem \ref{theorem-sigma}(i) does not hold with respect to the $J_1$-topology. Neither does it hold with respect to $J_2$, whereas $M_2$-convergence is a direct consequence of $M_1$-convergence, since $M_2$ is weaker than $M_1.$ 
\end{rem}

This paper is structured as follows. Section \ref{sec3} is devoted to various statistical applications of our limit theory.
  In Section \ref{sec2} we discuss properties of L\'evy integrals of predictable processes and recall essential estimates from \cite{KW} for those integrals.
 All proofs are 
demonstrated in Section~\ref{sec4}.

\section{Some statistical applications} \label{sec3}
\setcounter{equation}{0}
\renewcommand{\theequation}{\thesection.\arabic{equation}}

We start this section by giving an interpretation to the parameters $\alpha>0$ and $\beta \in (0,2)$. Let us consider the linear fractional stable motion defined by
\begin{align} \label{flm1}
Y_t := c_0 \int_{\R} \{(t-s)_{+}^\alpha  - ( -s)_{+}^\alpha \}\, dL_s,
\end{align} 
where the constant $c_0$ has been introduced in assumption (A). It is well known that the process 
$(Y_t)_{t \geq 0}$ is well defined whenever $H=\alpha+1/\beta \in (1/2,1)$. Furthermore, the process
$(Y_t)_{t \geq 0}$ has stationary symmetric $\beta$-stable increments, H\"older continuous paths of all orders smaller than $\alpha$ and self-similarity index $H$, i.e.
\[
\left( Y_{at} \right) _{t \geq 0} \eqschw \left(a^{H}Y_{t} \right) _{t \geq 0}
\qquad \text{for any } a\in \R_+.
\]   
We refer to e.g. \cite{BCI} for more details. As it has been discussed in \cite{BLP,BP15} in the setting
$\sigma=1$, the small scale behaviour of the process $X$ is well approximated by the corresponding
behaviour of the linear fractional stable motion $Y$. In other words, when the intermittency process
$\sigma$ is smooth, we have that 
\[
X_{t+\Delta} - X_t \approx \sigma_t(Y_{t+\Delta} - Y_t)
\]   
for small $\Delta>0$. Thus, intuitively speaking, the properties of $Y$ (H\"older smoothness, self-similarity) transfer to the process $X$ on small scales. 

Having understood the role of the parameters $\alpha>0$ and $H=\alpha+1/\beta \in (1/2,1)$ from 
the modelling perspective, it is obviously important to investigate estimation methods for these parameters.  We note that the conditions $\alpha>0$ and $H \in (1/2,1)$  imply the
restrictions  $\beta \in (1,2)$ and $\alpha < 1-1/\max\{p,\beta\}$. Hence, the regime of  Theorem~\ref{theorem-sigma} (iii) is never applicable.

We start with a direct estimation procedure, which identifies the convergence 
rates in Theorem \ref{theorem-sigma} (i)-(ii).  We apply these convergence results only for 
$t=1$ and $k=1$.
For $p\in [\underline{p}, \overline{p}]$ with $\underline{p} \in (0,1)$ and
$\overline{p}>2$, we introduce the statistic 
\begin{align} \label{statistic}
S_{\alpha, \beta} (n,p):= -\frac{\log V(p)^n}{\log n} \qquad \text{with} \qquad V(p)^n= V(p;1)^n_1.
\end{align} 
When the underlying L\'evy motion $L$ is symmetric $\beta$-stable and the assumptions of 
Theorems~\ref{theorem-sigma} (i)-(ii) are satisfied, we obtain that
\begin{align} \label{logscale}
S_{\alpha, \beta} (n,p) \toop  S_{\alpha, \beta} (p)  :=
\left \{
\begin{array} {ll}
\alpha p: & \alpha < 1-1/p \text{ and } p>\beta \\
pH-1: &\alpha < 1-1/\beta \text{ and } p<\beta 
\end{array} 
\right. 
\end{align}
Indeed, the result of Theorem~\ref{theorem-sigma} (i) shows that 
\[
\frac{\alpha p \log n + \log V(p)^n}{\log n} \stab 0 \qquad \Rightarrow \qquad
\frac{\alpha p \log n + \log V(p)^n}{\log n} \toop 0.
\]
This explains the first line in  \eqref{logscale}. Similarly, Theorem~\ref{theorem-sigma} (ii) 
implies the second convergence result of \eqref{logscale}. At this stage we remark that the limit 
$S_{\alpha, \beta} : [\underline{p}, \overline{p}] \setminus\{\beta\} \to \R$ 
is a piecewise linear function with two different slopes. It can be continuously extended to the function 
$S_{\alpha, \beta} : [\underline{p}, \overline{p}]  \to \R$, whose definition can be further extended 
to include all values 
\[
(\alpha, \beta) \in J:=\left \{(\alpha, \beta)\in \R^2:~\beta \in [1,2], ~\alpha \in [0,1-1/\beta] \right\}.
\]   
Let $(\alpha_0, \beta_0) \in J^{\circ}$, where $J^{\circ}$ is the set of all inner points of $J$, denote the true parameter of the model \eqref{lss}. Now, it is natural to consider the $L^2$-distance between the observed scale function $S_{\alpha_0, \beta_0} (n,p)$ and the theoretical limit $S_{\alpha, \beta} (p)$:
\begin{align} \label{scaleest}
( \hat{\alpha}_n, \hat{\beta}_n) \in \text{argmin}_{(\alpha,\beta) \in J} 
\|S_{\alpha_0, \beta_0} (n) -   S_{\alpha, \beta} \|_{L^2([\underline{p}, \overline{p}])} 
\end{align} 
with $S_{\alpha_0, \beta_0} (n):= S_{\alpha_0, \beta_0} (n, \cdot)$. This approach is somewhat similar to the estimation method proposed in \cite{glt15}. We notice that, for a finite $n$,
the minimum of the $L^2([\underline{p}, \overline{p}])$-distance at  \eqref{scaleest} is not necessarily 
obtained at a unique point, and we take an arbitrary measurable minimiser 
$( \hat{\alpha}_n, \hat{\beta}_n)$. Our next result 
shows consistency of the estimator $( \hat{\alpha}_n, \hat{\beta}_n)$.

\begin{cor} \label{cor1}
Let $(\alpha_0, \beta_0) \in J^{\circ}$ and let $L$ be a symmetric $\beta$-stable 
L\'evy motion. Assume that the conditions of Theorem~\ref{theorem-sigma} (i) (resp. Theorem~\ref{theorem-sigma} (ii)) hold when $\alpha\in (0,1-1/p)$ and $p>\beta$ (resp. 
$\alpha\in (0,1-1/\beta)$ and $p<\beta$).  Then we obtain convergence in probability
\[
( \hat{\alpha}_n, \hat{\beta}_n) \toop (\alpha_0, \beta_0). 
\] 
\end{cor}
\begin{proof}
Set $r_0=(\alpha_0, \beta_0)$ and $\widehat{r}_n=( \hat{\alpha}_n, \hat{\beta}_n)$. We first show
the convergence
\begin{align} \label{pnconvergence}
\|S_{r_0} (n) -   S_{r_0} \|_{L^2([\underline{p}, \overline{p}])} \toop 0.
\end{align}  
From  \eqref{logscale} we deduce that $S_{r_0} (n,p) \toop S_{r_0} (p)$ for all 
$p \in [\underline{p}, \overline{p}] \setminus\{\beta\}$. Furthermore, for any $p \in [\underline{p}, \overline{p}]$, it holds that 
\[
\left(V(\overline{p})^n \right)^{1/\overline{p}} \leq \left(V(p)^n \right)^{1/p} \leq
\left(V(\underline{p})^n \right)^{1/\underline{p}}.
\]
Hence, we deduce the inequality 
\[
\left| \frac{\log V(p)^n}{\log n} \right| \leq \max \left\{\frac{p}{\overline{p}} \cdot \left|
 \frac{\log V(\overline{p})^n}{\log n} \right|,  \frac{p}{\underline{p}} \cdot \left|
 \frac{\log V(\underline{p})^n}{\log n} \right|  \right\}.
\] 
Since $|\log V(\underline{p})^n/\log n| \toop pH-1$ and $|\log V(\overline{p})^n/\log n| \toop \alpha
p$, because $\underline{p}<1<\beta$ and $\overline{p}>2>\beta$, we readily deduce the convergence at \eqref{pnconvergence} by dominated convergence theorem.

Now, we note that the mapping $G: J \to G(J)
\subset L^2([\underline{p}, \overline{p}])$, $r \mapsto S_r$, is a homeomorphism. Thus, it suffices to 
prove that $\|S_{\widehat{r}_n} -   S_{r_0} \|_{L^2([\underline{p}, \overline{p}])} \toop 0$ to conclude 
$\widehat{r}_n \toop r_0$. To show the latter we observe that 
\begin{align}
\|S_{\widehat{r}_n} -   S_{r_0} \|_{L^2([\underline{p}, \overline{p}])} & \leq
\|S_{r_0}(n) -   S_{r_0} \|_{L^2([\underline{p}, \overline{p}])} 
+ \|S_{r_0}(n) -   S_{\widehat{r}_n} \|_{L^2([\underline{p}, \overline{p}])} \\
&= \|S_{r_0}(n) -   S_{r_0} \|_{L^2([\underline{p}, \overline{p}])} 
+\min_{r \in J} \|S_{r_0}(n) -   S_{r} \|_{L^2([\underline{p}, \overline{p}])} \\
& \leq 2 \|S_{r_0}(n) -   S_{r_0} \|_{L^2([\underline{p}, \overline{p}])} \toop 0.
\end{align} 
This completes the proof of Corollary \ref{cor1}.
\end{proof}

In practice  the integral in \eqref{scaleest} needs to be discretised.  
We further remark 
that the estimator  $S_{\alpha, \beta} (n,p)$ has the rate of convergence $\log n$ due to the bias 
$V(p)/\log n$, where $V(p)$ denotes the limit of $V(p)^n$.    

As for the estimation of the self-similarity parameter  $H=\alpha+1/\beta \in (1/2,1)$,
there is an alternative estimator based on a ratio statistic. Recalling that  $\beta \in (1,2)$, 
we deduce for any $p \in (0,1]$
\begin{align} \label{ratiostatistic}
R(n,p) := \frac{\sum_{i=2}^n |X_{\frac{i}{n}} - X_{\frac{i-2}{n}}|^p}{\sum_{i=1}^n |X_{\frac{i}{n}} - X_{\frac{i-1}{n}}|^p} \toop 2^{pH}  
\end{align}
by a direct application of Theorem \ref{theorem-sigma} (ii). Thus, we immediately conclude that 
\[
\hat{H}_n := \frac{\log R(n,p)}{ p \log 2} \toop H. 
\]
This type of idea is rather standard in the framework of a fractional Brownian motion with Hurst 
parameter $H$. It has been also applied for Brownian semi-stationary processes in \cite{BCP11,BCP13}. Theorem 1.2 (i) in \cite{BLP}, which has been shown in the setting $\sigma=1$,
suggests that the statistic  $\hat{H}_n$ has convergence rate $n^{1 - 1/(1-\alpha)\beta}$
whenever $p \in (0,1/2]$. Furthermore, the rate of convergence can be improved to $\sqrt{n}$ 
via using $k$th order increments with $k\geq 2$ (cf. \cite[Theorem 1.2 (ii)]{BLP}).  
However, 
we dispense with the precise proof of these statements for 
non-constant intermittency process $\sigma$.

\begin{rem}  \rm 
We remark that the linear fractional stable motion $(Y_t)_{t \geq 0}$ is well defined for 
$H=\alpha+1/\beta \in (0,1)$ and $\alpha \in (-1/\beta, 1-1/\beta)$. In this case the process $Y$ has 
unbounded paths whenever $\alpha<0$. Since in this framework there is no a priori lower bound 
on the parameter $\beta$, it is hard to apply Theorem \ref{theorem-sigma} (ii) to estimate the parameter $H$, because it requires the condition $p<\beta$. An elegant solution of this problem has been found 
in a recent work \cite{DI} in the context of  a linear fractional stable motion. It turns out that in this 
setting the asymptotic result of Theorem \ref{theorem-sigma} (ii) remains valid for powers $p \in (-1,0)$. Hence, 
it holds that 
\[
\hat{H}_n  \toop H \qquad \text{for } p \in (-1,0) 
\]
when the underlying process is a linear fractional stable motion. However, proving this result for 
a general L\'evy semi-stationary process is a much more delicate issue.  
\end{rem}  

Another important object for applications in turbulence modelling is the intermittency process $\sigma$.
First of all, we remark that the process  $\sigma$ in the general model \eqref{lss} is statistically not 
identifiable. This is easily seen, because multiplication of $\sigma$ by a constant can not be distinguished from the multiplication of, say, L\'evy process $L$ by the same constant. However, 
it is very well possible to estimate the \textit{relative intermittency}, which is defined as
\begin{align}
RI(p):= \frac{\int_0^t |\sigma_s|^p ds }{\int_0^1 |\sigma_s|^p ds }, \qquad t \in (0,1), 
\end{align}     
for $p \in (0,1]$. The relative intermittency, which has been introduced in \cite{BPS} for
$p=2$ in the context of Brownian semi-stationary processes, describes the relative amplitude of the velocity process on an interval $[0,1]$. Applying
the convergence result of Theorem \ref{theorem-sigma} (ii) for $p \in (0,1]$, the relative intermittency can be consistently estimated via 
\begin{align}
RI(n,p):= \frac{V(p)^n_t}{V(p)^n_1}  \toop RI(p). 
\end{align}   
Again we suspect that the associated convergence rate is $n^{1 - 1/(1-\alpha)\beta}$ whenever
$p \in (0,1/2]$ as suggested
by \cite[Theorem 1.2 (i)]{BLP}.

\section{Preliminaries: Estimates on L\'evy integrals} \label{sec2}
\setcounter{equation}{0}
\renewcommand{\theequation}{\thesection.\arabic{equation}}

To prove the various limit theorems we need very sharp  estimates of the $p$th moments of the increments of process $X$ defined in \eqref{lss}. In fact, we need such estimates for several different  processes related to $X$ obtained by different   truncations. Below we explain some intuition behind the techniques we 
use to estimate the $p$th moments of $X$.  Recall that if $B$ is a Brownian motion and $F$ is predictable process then we 
have the following estimate: For any $q>0$ there exists a finite constant  $C$,  only depending on $p$, such that 
\begin{equation}\label{est-Brow}
 \E\Big[ \Big| \int_0^t F_s\,dB_s\Big|^q\Big]\leq C \E\Big[ \Big(\int_0^t F_s^2\,ds\Big)^{q/2}\Big] = C \E\big[\| F \|_{L^2([0,t])}^q\big],
\end{equation}
which follows by the 
Burkholder-Davis-Gundy inequality, see~\cite[Theorem~3.28]{K-S}.  The estimate \eqref{est-Brow} is crucial for proving  limit theorems when the driving process is a Brownian motion, see e.g.\ \cite{BCP11}. But the situation becomes more complicated when the Brownian motion $B$ is replaced by a L\'evy process $L$ as considered in the present paper.   For integrals with respect to general L\'evy processes $L$ we cannot estimate the sample paths $s\mapsto F_s(\omega)$, $\omega\in \Omega$, in the $L^2([0,t])$-norm.  We need to consider other functionals,  which depend on the L\'evy measure $\nu$ of $L$. 
When $F:\R_+\to \R$ is a deterministic function and $L$ is a L\'evy process, such estimates go back to Rajput and Rosi\'nski~\cite[Theorem~3.3]{RajRos}. Their results imply the existence of a constant $C>0$ such that   
$$\E[ | \int_0^t F_s\,dL_s|^q]\leq C \| F \|_{L,q}^q, $$ 
where $\| \cdot \|_{L,q}$ is a certain functional to be defined below (when $L$ is symmetric and without Gaussian component). The decoupling approach used in Kwapi\'en and 
Woyczy\'nski~\cite{KW} provides an extension of the results  to 
general predictable $F$, see  Lemmas~\ref{KwaWoyIso} and \ref{WeaLBetIso} below. These results can be thought of as extensions of \eqref{est-Brow} to integrals with respect to L\'evy processes. Before stating the results precisely, we need the following notation.

Let $L=(L_t)_{t\in \R}$ be a symmetric L\'evy process on the real line with $L_0=0$, L\'evy measure $\nu$ and without a Gaussian component. 
For a predictable process $(F_t)_{t\in\R}$ and for $q=0$ or $q\geq 1$ we define 
\[ \Phi_{q,L}(F) :=\int_{\R^2} \phi_{q}(F_su) \diff s\, \nu (d u),\]
where 
\[\phi_q(x):=|x|^q\mathds 1_{\{|x|>1\}}+x^2\mathds 1_{\{|x|\leq 1\}}.\]
A predictable process $F=(F_t)_{t\in \R}$ is integrable with respect to $(L_t)_{t\in\R}$ in the sense of \cite{KW} if and only if $\Phi_{0,L}(F)<\infty$ almost surely (cf.\ \cite[Theorem 9.1.1]{KW}). The linear space of predictable processes satisfying $\Phi_{q,L}(F)<\infty$ will be denoted by \ldl{q}{L}. In order 
to estimate the $p$-moments of stochastic integrals we introduce for all  $q\geq 1$ 
\begin{equation}
 \label{ljsdfljsdf}\|F\|_{q,L}:=\inf\{\lambda\geq 0\,:\, \Phi_{q,L}(F/\lambda)\leq 1\},\qquad F\in \ldl{q}{L}.
\end{equation}
The following two results from Chapter 9.5 in \cite{KW} will play a key role for our proofs.

\begin{lem}[\cite{KW}, Equation~(9.5.3)]\label{KwaWoyIso}
For all $q\geq 1$ there is a constant $C,$ depending only on $q$, such that we obtain for all $F\in\ldl{q}{L}$
\begin{equation}\label{ewrwera}
 \E\bigg[\bigg|\int_\R F_s \diff L_s\bigg|^q\bigg]\leq C\E\big[\|F\|^q_{q,L}\big].
\end{equation}
\end{lem}

The above lemma follows by \cite[Equation~(9.5.3)]{KW} and the comments following it.
 Actually, \cite[Equation~(9.5.3)]{KW} only treats the case where the stochastic  integral in \eqref{ewrwera} is over a finite time interval, say
 $\int_0^t F_s\,dL_s$. However, the definition of the stochastic integral and the estimates of the  integral in \cite[Chapters~8--9]{KW} extend to the case  of $\int_\R F_s\,dL_s$ in a natural way. For example,  the set function $m((s,t]) = L_t- L_s$ for $s<t$, extends only to a  $\sigma$-finite stochastic measure  defined on the $\delta$-ring of bounded Borel subsets of $\R$ (cf. \cite[Theorem~8.3.1]{KW}) and so forth.  A similar extension applies to the results from \cite{KW} mentioned below.

For the next result, which is an immediate consequence of \cite[Theorem 9.5.3]{KW}, we use the notation $\|Z\|^\beta_{\beta,\infty}=\sup_{\lambda>0}\lambda^\beta \P[|Z|>\lambda]$ for an arbitrary  random variable $Z$. For $q<\beta$ it holds that 
$$\E[|Z|^q]^{1/q}\leq \|Z\|_{\beta,\infty}\leq\E[|Z|^\beta]^{1/\beta}.$$ 
In the literature, $\|\cdot\|_{\beta,\infty}$ is often referred to as the weak $L^\beta$-norm. However, $\|\cdot\|_{\beta,\infty}$ generally fails to satisfy the triangle  inequality.

\begin{lem}[\cite{KW}, Theorem 9.5.3] \label{WeaLBetIso}
Let $(L_t)_{t\in\R}$ be a symmetric $\beta$-stable L\'evy process. Then there is a positive constant $C>0$ such that for all $(F_t)_{t\in\R}$ in \ldl{0}{L} it holds that
\begin{align}\label{WeaLBetIne}
\bigg\|\int_\R F_s \diff L_s\bigg\|^\beta_{\beta,\infty}\leq C\E\bigg[\int_\R|F_s|^\beta \diff s\bigg].
\end{align}
\end{lem}

The next remark gives sufficient conditions for  the process $X$ 
introduced at \eqref{lss} to be well-defined.

\begin{rem}\label{Xintegrable}
Suppose that (A) is satisfied and define the two processes $F^{(1)}$ and $F^{(2)}$ by $F^{(1)}_s = (g(-s)-g_0(-s))\sigma_s$ and $F^{(2)}_s = g'(-s) \sigma_s$ for $s<0$.
Then the process  $X$ given by \eqref{lss} is well-defined if there exists a $\beta'>\beta$ such that 
\begin{equation}\label{eq:23234555}
\int_{-\infty}^{-\delta} \Big(| F^{(i)}_s|^\theta \1_{\{| F^{(i)}_s|\leq 1\}}+  | F^{(i)}_s|^{\beta'} \1_{\{| F^{(i)}_s|> 1\}}\Big)\,ds<\infty
\end{equation}
almost surely for $i=1,2$. 
%
To show the above we argue as follows:  For any $\beta'\in(\beta,2]$ we deduce from (A) 
and simple  calculations  the estimate 
 \begin{align}\label{Phi_0Est}
\int_\R \big(|ux|^2\wedge 1\big)\nu(d x)\leq C(|u|^{\theta}\mathds 1_{\{|u|\leq 1\}}+|u|^{\beta'}\mathds 1_{\{|u|>1\}}),\qquad u\in \R. 
\end{align}
Then, an application of the mean value theorem combined with  assumption \eqref{eq:23234555} yields that $\Phi_{0,L}(H^{(t)})<\infty$ almost surely for all $t>0$, where $H^{(t)}_s=(g(t-s)-g_0(-s))\sigma_{s}$. This guarantees the existence of  the process $X$ due to \cite[Theorem 9.1.1]{KW}.
\end{rem}
In our proofs we will need the following properties of the functional $\| \cdot \|_{L,q}$ defined in \eqref{ljsdfljsdf}. 
\begin{rem}
 The functional  $\| \cdot \|_{L,q}$ satisfies the following three properties: 
\begin{enumerate}[(i)]
 \item  \label{prop-1-aaaa} Homogeneity: For all $\lambda\in\R$, $F\in \ldl{q}{L}$, $\|\lambda F\|_{q,L}=|\lambda|\|F\|_{q,L}$.
 \item  \label{prop-2-aaaa} Triangle inequality (up to a constant): There exists a constant $C>0$ such that for all $m\in\N$ and $F^1,...,F^m\in\ldl{q}{L}$ we have 
\begin{align}\label{triangleine}
\|F^1+\dots +F^m\|_{q,L}\leq C\big(\|F^1\|_{q,L}+\dots+\|F^m\|_{q,L}\big).
\end{align}
\item \label{prop-3-aaaa} Upper bound: For all $F\in \ldl{q}{L}$ we have 
\begin{align}\label{NormVsPhiIneq}
\|F\|_{q,L}\leq \Phi_{q,L}^{1/2}(F)\vee \Phi^{1/q}_{q,L}(F).
\end{align}
\end{enumerate}
Property \eqref{prop-1-aaaa} follows directly from the definition of $\|\cdot\|_{L,q}$ in \eqref{ljsdfljsdf}. To show property \eqref{prop-2-aaaa} it is sufficient to derive \eqref{triangleine} for $F^1,...,F^m\in\ldlnr{q}{L}$, where $\ldlnr{q}{L}$ denotes the subspace of nonrandom processes in \ldl{q}{L}. We will show that there is a norm $\|\cdot \|'_{q,L}$ on $\ldlnr{q}{L}$ and $c>0$ and $C>0$ such that
\[c\|F \|'_{q,L}\leq \|F \|_{q,L}\leq C\|F \|'_{q,L},\quad\text{for all }F\in\ldlnr{q}{L},\]
which then implies \eqref{triangleine}.
To this end, let
\begin{align*}
\widetilde\phi_q(x)&:=(2/q|x|^q+1-2/q)\mathds 1_{\{|x|>1\}}+x^2\mathds 1_{\{|x|\leq 1\}}.
 \end{align*}
 Clearly, there exist $c,C>0$ such that
 \[c\widetilde\phi_{q}(x)\leq \phi_{q}(x)\leq C\widetilde\phi_{q}(x)\quad\text{for all }x\in\R.   \]
 Moreover, since the function $\widetilde\phi_q$ is convex, the functional
 \[\|F\|'_{q,L}=\inf\bigg\{\lambda\geq 0\,:\, \int_{\R^2} \widetilde\phi_{q}(F_su/\lambda) \diff s\, \nu (d u)\leq 1\bigg\}\]
 is a norm on $\ldlnr{q}{L}$, called the Luxemburg norm (cf. \cite[Chapter 1]{Mus}). Using convexity of $\widetilde\phi_q$ it follows by straightforward calculations that $c\|F \|'_{q,L}\leq \|F \|_{q,L}\leq C\|F \|'_{q,L}$ for all $F\in\ldlnr{q}{L}$. This implies \eqref{triangleine}. Finally, property \eqref{prop-3-aaaa} follows by the fact that  $\phi_q(\lambda x)\leq (\lambda^2\vee \lambda^q)\phi_q(x)$ for all $\lambda\geq 0$.
 \end{rem}

 We conclude this subsection with a remark on the situation when the integrator is a non-symmetric L\'evy process $(\wt L_t)_{t\in\R}$ with $\wt L_0=0$,  L\'evy measure $\wt\nu$, shift parameter $\eta$, without a Gaussian part, and the truncation function  $\tau\! : x\mapsto \mathds 1_{\{|x|<1\}}+\text{sign}(x)\mathds 1_{\{|x|\geq 1\}}$. That is, 
 for all $\theta\in \R$, 
 \begin{equation}
 \E[ e^{i \theta \tilde L_1} ] = \exp\Big( i\theta  \eta + \int_\R \big( e^{i\theta x} -1-i\theta \tau(x)\big)\,\wt \nu(dx)\Big).
 \end{equation}
For a predictable process $(F_t)_{t\in\R}$ define 
\[\Psi_{0,\wt L}(F)=\Big| \int_{\R^2} \tau(uF_s)-\tau(u)F_s \diff s\, \wt\nu(du) + \eta F_s \Big|. \]
Then, the condition
\begin{equation}\label{eqrlqwe}
 \Phi_{0,\wt L}(F)+\Psi_{0,\wt L}(F)<\infty \quad \text{almost surely}
\end{equation}
is sufficient for the integral $\int_\R F_s \diff \wt L_s$ to exist, and we write $F\in\ldl{0}{\wt L}$. Indeed, this is a consequence of \cite[Theorem~9.1.1 and  pp.~217--218]{KW} combined with the estimate 
 \cite[Lemma~2.8]{RajRos}.

\section{Proofs}\label{sec4}
\setcounter{equation}{0}
\renewcommand{\theequation}{\thesection.\arabic{equation}}

In this section we present the proofs of our main results.  Let us first briefly  comment on some of the techniques used  in the proofs of  Theorem~\ref{theorem-sigma}.

The proof of Theorem~\ref{theorem-sigma} (i) is  divided into two parts. First we show the theorem under the assumption that $L$ is a compound Poisson process with jumps bounded away from zero in absolute value by some $a>0$. In this situation,  (B1) ensures that the integral in \eqref{lss} can be defined $\omega$ by $\omega$, and the limit of $V(p;k)_t^n$ can be derived by similar means as in \cite{BLP}.
Thereafter, we argue that the contribution of the jumps of $L$ with absolute value $\leq a$ to the power variation becomes negligible as $a\to 0.$ 
The proof of Theorem~\ref{theorem-sigma} (ii)   relies on Bernstein's blocking technique combined with Theorem~\ref{maintheorem} (ii) and several approximation steps.  
A key step in the proof of Theorem~\ref{theorem-sigma} (iii) is an application of 
a suitable stochastic Fubini result. For this purpose we present and prove a stochastic Fubini theorem for L\'evy integrals with predictable integrands that is applicable under our assumptions.

Throughout the proofs we denote all positive constants that do not depend on $n$ or $\omega$ by $C$, eventhough they may change from line to line.
Similarly, we will denote by $K$ any positive random variable that does not depend on $n$, but may change from line to line.
For a random variable $Y$ and $q>0$ we denote $\|Y\|_q=\E[|Y|^q]^{1/q}.$ 
We frequently use the notation
\[g_{i,n}(s)= \sum_{j=0}^k (-1)^{j}\binom{k}{j}g((i-j)/n-s),\]
which allows us to express the $k$-th order increments of $X$ as
\[\din X=\int_{-\infty}^{i/n} g_{i,n}(s)\sigma_{s-}\diff L_s.\]
Recalling that $|g^{(k)}(s)|\leq C t^{\alpha-k}$ for all $s\in (0,\delta)$ and  $g^{(k)}$ is decreasing on $(\delta,\infty)$ by assumption (A) , Taylor expansion leads to the following important estimates.
\begin{lem} \label{g_{i,n}Est}
Suppose that assumption (A) is satisfied. It holds that
\begin{align*}
|g_{i,n}(s)|&\leq C(i/n-s)^\alpha\qquad\text{for }s\in[(i-k)/n,i/n],\\
|g_{i,n}(s)|&\leq Cn^{-k}((i-k)/n-s)^{\alpha-k}\qquad\text{for }s\in(i/n-\delta,(i-k)/n),\text{ and}\\
|g_{i,n}(s)|&\leq Cn^{-k}\big(\mathds 1_{[(i-k)/n-\delta,i/n-\delta]}(s)+g^{(k)}((i-k)/n-s)\mathds 1_{(-\infty,(i-k)/n-\delta)}(s)\\
&\text{for }s\in(-\infty, i/n-\delta].
\end{align*}
\end{lem}

\begin{proof}
The  first inequality  follows directly from condition \eqref{kshs} of (A). The second inequality
is a straightforward consequence of Taylor expansion of order $k$ and the condition
$|g^{(k)}(t)|\leq K t^{\alpha-k}$ for $t\in (0,\delta)$. The third inequality  follows again through Taylor expansion and the fact that the function $g^{(k)}$ is decreasing on $(\delta, \infty)$.
\end{proof}

\begin{rem}\label{eqreassdfsfd}
 Throughout the proofs we will generally assume that the process $\sigma$ is uniformly bounded on $[-\delta,\infty)$. That is, there exists a deterministic constant $C>0$ such that $|\sigma_s|<C$ for all $s\geq -\delta.$ This does not restrict the generality of our results, since we can apply the following localisation argument.
Let $(S_m)_{m\geq 1}$ be a sequence of $(\mathcal F_t)_{t\geq -\delta}$-stopping times with $S_m\uparrow\infty$, such that $|\sigma_{s-}|\mathds 1_{\{S_m>-\delta\}}$ is bounded for all $s\in [-\delta, S_m].$ Let the process $\sigma^{(m)}$ be defined by 
\[
\sigma^{(m)}_s = \sigma_s \mathds 1_{\{ s<S_m\}} + \sigma_{S_m-}\mathds 1_{\{s\geq S_m>-\delta\}}.
\]
The process $\sigma^{(m)}$ is again c\`adl\`ag and adapted. We define the process $(X_t^{(m)})_{t\geq 0}$ by replacing $\sigma$ in the definition of $X$ by $\sigma^{(m)}$. We note that $(X_t^{(m)})$ is well-defined since $(X_t)$ is well-defined and that assumption (B1) and (B2) hold for $(X_t^{(m)})$ if they hold for $(X_t)$.
It holds that $\din X^{(m)}\mathds 1_{\{S_m>t\}}=\din X\mathds 1_{\{S_m>t\}}$ almost surely.
It is therefore sufficient to show that Theorem \ref{theorem-sigma} holds for the processes $X^{(m)}$. Then the theorem follows for the process $X$ by letting $m\to \infty$.
\end{rem}


\subsection{Proof of Theorem \ref{theorem-sigma} (i)}


For the proof of Theorem \ref{theorem-sigma} (i) we follow the strategy from \cite[Thm. 1.1 (i)]{BLP}. We assume first that $L$ is a compound Poisson process with jumps bounded in absolute value away from zero by some $a>0$. Later on, we argue that the small jumps of $L$ are asymptotically negligible. Recall that in order to show functional $\mathcal F$-stable convergence on $\mathbb D(\R_+;\R)$ it is sufficient to show $\mathcal F$-stable convergence on $\mathbb D([0,t_\infty];\R)$, for arbitrary but fixed $t_\infty>0$ (cf. \cite[Chapter 3.3]{Whitt}). Throughout this subsection we will therefore fix a $t_\infty>0$, and denote by $\mathbb D$ the space $\mathbb D([0,t_\infty];\R)$ equipped with the Skorokhod $M_1$-topology, and by $\skorcon$ the $\mathcal F$-stable convergence of $\mathbb D$-valued processes.
 
\subsubsection{Compound Poisson Case}\label{ComPoiCas}

Suppose that $(L_t)_{t\in\R}$ is a symmetric compound Poisson process with L\'evy measure $\nu$, satisfying $\nu([-a,a])=0$ for some $a>0.$
 Let $0\leq T_1<T_2<...$ denote the jump times of $(L_t)_{t\geq 0}$ in increasing order.
For $\eps>0$ we define
\begin{align*}
 \Omega_\eps=\big\{\omega \in\Omega : &\text{ for all $m$ with $T_m(\omega)\in[0,t_\infty]$ we have $|T_m(\omega)-T_{m-1}(\omega)|>\eps$}\\
&\text{ and $\Delta L_s(\omega)=0$ for all $s\in [-\eps,0]$}\big\}.
\end{align*}
We note that $\Omega_\eps\uparrow \Omega,$ as $\eps\downarrow 0.$ Letting
\[M_{i,n,\eps}:=\int_{i/n-\eps}^{i/n} g_{i,n}(s)\sigma_{s-}\diff L_s,\quad\text{and}\quad R_{i,n,\eps}:=\int^{i/n-\eps}_\infty g_{i,n}(s)\sigma_{s-}\diff L_s,\]
we have the decomposition $\din X=M_{i,n,\eps}+R_{i,n,\eps}.$ It turns out that $M_{i,n,\eps}$ is the asymptotically dominating term, whereas $R_{i,n,\eps}$ is negligible as $n\to\infty.$

We show that, on $\Omega_\eps$,
\begin{align}\label{MLim}
n^{\alpha p}\sum_{i=k}^{[nt]} |M_{i,n,\eps}|^p\skorcon Z_t,\quad\text{where}\quad Z_t=\sum_{m:T_m\in (0,t]}|\Delta L_{T_m}\sigma_{T_m-}|^p V_m,
\end{align}
where $(V_m)_{m\geq 1}$ are defined in Theorem \ref{theorem-sigma} (i). Denote by $i_m$ the random index such that $T_m\in((i_m-1)/n,i_m/n].$
Then, we have on $\Omega_\eps$
\[n^{\alpha p}\sum_{i=k}^{[nt]} |M_{i,n,\eps}|^p=V^{n,\eps}_t,\]
where
\[V^{n,\eps}_t= n^{\alpha p} \sum_{m:T_m\in (0,[nt]/n]}|\Delta L_{T_m}\sigma_{T_m-}|^p\left(\sum_{l=0}^{v_t^m}|g_{i_m +l,n}(T_m)|^p\right).\]
Here the random index $v^m_{t}$ is defined as
\[ v^m_{t }= v_t^m(\eps,n)=
\begin{cases} 
[\eps n]\wedge ([nt]-i_m) & \text{if } T_m-([\eps n]+i_m)/n>-\eps, \\
[\eps n]-1 \wedge ([nt]-i_m) & \text{if }T_m-([\eps n]+i_m)/n\leq-\eps.
\end{cases}
\]
Additionally, we set $v^m_{t }=\infty$ if $T_m>[nt]/n.$ We remark that $v_t^m$ attains the value $[nt]-i_m$ only if $T_m\in (t-\eps,t]$, which is the case for at most one $m$. 
 For the proof of \eqref{MLim} we first show stable convergence of the finite dimensional distributions of $V^{n,\eps}$. Thereafter, we show that the sequence $(V^{n,\eps})_{n\geq 1}$ is tight and deduce the functional convergence $V^{n,\eps}\skorcon Z.$

\begin{lem}\label{fidi}
For $r\geq 1$ and $0\leq t_1<\dots< t_r\leq t_\infty$ we obtain on $\Omega_\eps$ the $\mathcal F$-stable convergence
\[(V^{n,\eps}_{t_1},\dots, V^{n,\eps}_{t_r})\tols (Z_{t_1},\dots, Z_{t_r}),\quad \text{ as }n\to\infty.\]
\end{lem}
\begin{proof}
Let $(U_i)_{i\geq 1}$ be i.i.d. $\mathcal U([0,1])$-distributed random variables, defined on an extension $(\Omega',\mathcal F',\P')$ of the original probability space, independent of $\mathcal F.$ By arguing as in \cite[Section 5.1]{BLP}, we deduce for any $d\geq 1$ the $\mathcal F$-stable convergence
\[\{n^\alpha g_{i_m+l,n}(T_m)\}_{l,m\leq d}\tols \{h_k(l+U_m)\}_{l,m\leq d}\]
as $n\to\infty$, where $h_k$ is defined in \eqref{def-h-13}.
Defining
\begin{align*}
V_t^{n,\eps,d}&:=n^{\alpha p} \sum_{m\leq d:T_m\in (0,[nt]/n]}|\Delta L_{T_m}\sigma_{T_m-}|^p\left(\sum_{l=0}^{d}|g_{i_m +l,n}(T_m)|^p\right)\\
Z^d_t& := \sum_{m\leq d:T_m\in (0,t]}|\Delta L_{T_m}\sigma_{T_m-}|^p\left(\sum_{l=0}^{d}|h_k(l+U_m)|^p\right),
\end{align*}
the continuous mapping theorem for stable convergence yields
\begin{align}\label{fidiconsta}
(V^{n,\eps,d}_{t_1},\dots, V^{n,\eps,d}_{t_r})\tols (Z^d_{t_1},\dots, Z^d_{t_r}),\quad \text{ for }n\to\infty,
\end{align}
for all $d\geq 1.$  It follows by Lemma \ref{g_{i,n}Est} for all $l$ with $k\leq l<[n\delta]$ that
\[n^{\alpha p}|g_{i_m+l,n}(T_m)|^p\leq C|l-k|^{(\alpha-k)p},\]
where we recall that $(\alpha-k)p<-1.$
Consequently, we find a random variable $K>0$ such that for all $t\in[0,t_\infty]$ 
\[|V^{n,\eps,d}_t-V^{n,\eps}_t|\leq K\bigg(\sum_{m>d: T_m\in[0,t_\infty]} |\Delta L_{T_m} \sigma_{T_m-}|^p+ \sum_{m: T_m\in[0,t_\infty]}\sum_{l=v_t^m\wedge d}^\infty |l-k|^{(\alpha-k)p}\bigg).\]
By definition, the random index $v_t^m=v_t^m(n,\omega)$ satisfies $\liminf_{n\to\infty} v^m_t(n,\omega)=\infty$ for all $\omega$ with $T_m(\omega)\neq t.$ Consequently, we obtain that $\limsup_{n\to\infty}|V^{n,\eps,d}_t-V^{n,\eps}_t|\to 0$ almost surely as $d\to\infty.$ It follows that on $\Omega_\eps$
\begin{align} \label{V^nvsV^nd}
\limsup_{n\to\infty} \bigg\{\sup_{t\in\{t_1,\dots,t_r\}}|V_t^{n,\eps}-V_t^{n,\eps,d}|\bigg\}\to 0,\quad\text{almost surely, as }d\to\infty.
\end{align}
By monotone convergence theorem
 we obtain that $\sup_{t\in[0,t_\infty]}|Z^d_t-Z_t|\to 0$ as $d\to\infty.$ Together with \eqref{fidiconsta} and \eqref{V^nvsV^nd}, this implies the statement of the lemma by a standard approximation argument, see for example \cite[Thm 3.2]{Bill}.
\end{proof}

Recall that the stable convergence $V^{n,\eps}\skorcon Z$ is equivalent to the convergence of $(V^{n,\eps},X)\tol (Z,X)$ for all $\mathcal F$-measurable random variables $X$, cf. \cite[Prop. 5.33]{JacShir}. Consequently, Lemma \ref{fidi} and the following lemma together with Prokhorov's theorem imply \eqref{MLim}, where we recall that $\mathbb D([0,t_{\infty}])$ equipped with the $M_1$ topology is a Polish space.

 \begin{lem}\label{tight}
The sequence of $\mathbb D$-valued processes $(V^{n,\eps})_{n\geq 1}$ is tight.
\end{lem}
\begin{proof}
It is sufficient to show that the conditions of \cite[Theorem 12.12.3]{Whitt} are satisfied.
Condition (i) is satisfied, since the family of real valued random variables $(V_{t_\infty}^{n,\eps})_{n\geq 1}$ is tight by Lemma \ref{fidi}. Condition (ii) is satisfied, since the oscillating function $w_s$ introduced in \cite[chapter 12, (5.1)]{Whitt} satisfies $w_s(V^{n,\eps},\theta)=0$ for all $\theta>0$ and all $n$, since $V^{n,\eps}$ is increasing.
\end{proof}


This concludes the proof of \eqref{MLim}. Next we show that 
\begin{align}\label{TaiComPoi}
n^{\alpha p}\sum_{i=k}^{[nt_\infty]} |R_{i,n, \eps}|^p \toop 0.
\end{align}
Recalling that $\alpha<k-1/p$, it is sufficient to show that
\[\sup_{n\in\N,\ i\in\{k,\dots,[nt_\infty]\}}n^{k}|R_{i,n,\eps}|<\infty,\quad \text{almost surely.}\]
It follows from Lemma \ref{g_{i,n}Est} that
\[n^k |g_{i,n}(s)\sigma_{s-}|\leq C (\mathds 1_{[-\delta,t_\infty]}(s)+|g^{(k)}(-s)\sigma_{s-}|\mathds 1_{(-\infty,-\delta)}(s)):=\psi_s.\]
Let $\wt L = (\wt L_t)_{t\in\R}$ denote the total variation process defined as 
$\wt L_0=0$ and $\wt L_t -\wt L_u$ is the total variation of $v\mapsto L_v$ on $(u,t]$  for all $u<t$. Since $L$ is a compound Poisson process, the process $\wt L$ is well-defined, finite and we deduce from \cite[Theorem~21.9]{Sato} that  $\wt L$ is a L\'evy process with L\'evy measure  $\wt \nu = 2 \nu_{|\R_+}$ and shift parameter $\eta$ with respect to the truncation function $\tau\! : x\mapsto \mathds 1_{\{|x|<1\}}+\text{sign}(x)\mathds 1_{\{|x|\geq 1\}}$ given by $\eta = \int_\R \tau(x)\,\wt \nu(dx)$. 
Next we use the following estimate:
\[n^{k}|R_{i,n,\eps}|\leq \int_{(-\infty,\frac i n- \eps]} n^k |g_{i,n}(s)\sigma_{s-}|\,d  \wt  L_s \leq \int_{\R} \psi_s\, d \wt L_s .\]
The right-hand side is finite almost surely due to  the following Lemma~\ref{poiintbou}, and the proof of \eqref{TaiComPoi} is complete.

\begin{lem}\label{poiintbou} Let $L$ be a symmetric compound Poisson process with L\'evy measure $\nu$ satisfying $\nu([-a,a])=0$ for some $a\in(0,1]$ and let  $\wt L$ and $\psi$ be  given as above. Suppose, in addition,  that  (B1) is satisfied.   Then the stochastic integral $\int_\R \psi_s\,d\wt L_s$ exists and is finite almost surely. 
\end{lem}
\begin{proof}
To show that the stochastic integral $\int_\R \psi_s\,d\wt L_s$ is well-defined it is
 enough to prove that $\Phi_{0,\wt L}(\psi)+\Psi_{0,\wt L}(\psi)<\infty$ almost surely (see \eqref{eqrlqwe} of Section~\ref{sec2}).
For some $\beta'>\beta$  we have from (B1) that 
\[\int_\R|\psi_s|^{\theta}\mathds 1_{\{|\psi_s|\leq 1\}}+|\psi_s|^{\beta'}\mathds 1_{\{|\psi_s|>1\}}\diff s<\infty,\quad\text{a.s.}\]
This implies that $\Phi_{0,\wt L}(\psi)<\infty$ almost surely (cf.\ \eqref{Phi_0Est}). 
Next we note that 
\begin{align}
\Psi_{0,L}(\psi) ={}&   \Big|  \int_{\R^2} \tau(x\psi_s)-\tau(u)\psi_s \diff s\, \wt\nu(dx) + \eta \psi_s \Big| 
 =  \Big|  \int_{\R^2} \tau(x\psi_s) \diff s\, \wt\nu(dx) \Big| ,
\end{align}
where the second equality follows by definition of $\eta$ above. Hence, to show that  $\Psi_{0,L}(\psi)<\infty$ almost surely,  it suffices according to (B1) to  derive the following estimate.
There exists a constant $C>0$ such that for all $u\in\R$
\begin{align}\label{Psiest}
 \int_\R |\tau(ux)|\,\wt\nu(dx)\leq C\big(|u|^{\rho}\mathds 1_{\{|u|\leq 1\}}+\mathds 1_{\{|u|> 1\}}\big). 
\end{align}
where $\rho$ is as in assumption (B1).
By the definitions of $\tau$ and $\wt \nu$ we have that 
\begin{equation}\label{eqrljadda}
\int_\R |\tau(ux)|\,\wt\nu(dx)= | u | \int_{\{|x| \leq |u|^{-1}\}} |x| \,\nu(dx) + \nu\big(x\in\R: |xu|>1\big). 
\end{equation}
We recall that $\limsup_{t\to\infty}\nu([t,\infty))t^\theta<\infty.$ Since $\nu$ is finite, there exists $C_0>0$ such that $\nu([t,\infty))\leq C_0/t^\theta$ for all $t\geq a$. Consequently, we obtain for all $t\geq a$ and $f(u)=\mathds 1_{[t,\infty)}(u)$ 
\[\int_a^\infty f(x)\,\nu(dx)\leq \frac{C_0}{\theta}\int_a^\infty f(x)x^{-\theta-1}\diff x.\] 
By monotone approximation, the inequality remains valid for all nondecreasing $f:[a,\infty)\to \R_+.$ 
We estimate the first term on the right-hand side of \eqref{eqrljadda} as follows:
\begin{align}
 | u | \int_{\{|x| \leq |u|^{-1}\}} |x| \,\nu(dx) {}& \leq (C_0/\theta)\1_{\{|u|\leq a^{-1}\}} |u| \int_a^{|u|^{-1}} |x|^{-\theta}\,dx
\\ {}& 
\leq C\1_{\{|u|\leq a^{-1}\}}   \begin{cases}
|u|^{\theta} & \qquad \theta< 1 \\ 
 |u| (\log(1/|u|)+ \log(1/a)) & \qquad \theta = 1\\ 
 |u|& \qquad \theta >1. 
\end{cases}
\end{align}
For the second term on the right-hand side of \eqref{eqrljadda} we use the following estimate
\begin{equation}
\nu\big(x\in\R: |xu|>1\big)\leq C (\1_{\{|u|> 1\}}+(|u|^{-1})^{-\theta}\1_{\{|u|\leq 1\}})= C (\1_{\{|u|>1\}}+|u|^{\theta}\1_{\{|u|\leq 1\}})
\end{equation}
for all $u\in\R$, which completes the proof of \eqref{Psiest} and hence of the lemma.  
\end{proof}

Recalling the decomposition $\din X=M_{i,n,\eps}+R_{i,n,\eps}$ we obtain by Minkowski's inequality 
\[\sup_{t\in[0,t_\infty]}\bigg|\big(n^{\alpha p}V(p;k)^n_t\big)^{\frac 1 p}-\bigg(n^{\alpha p}\sum_{i=k}^{[nt]} |M_{i,n,\eps}|^p\bigg)^{\frac 1 p}\bigg|\leq\bigg(n^{\alpha p}\sum_{i=k}^{[nt_\infty]} |R_{i,n,\eps}|^p\bigg)^{ \frac 1 p}.\]
Therefore, by virtue of \eqref{MLim} and \eqref{TaiComPoi}, we conclude that
\[n^{\alpha p}V(p;k)^n_t\skorcon Z_t \quad\text{on $\Omega_\eps.$}\]
By letting $\eps\to 0$ we conclude that Theorem \ref{theorem-sigma} (i) holds, when $L$ is a compound Poisson process with jumps bounded away from 0.

\subsubsection {Decomposition into big and small jumps}\label{SmaJum}

In this section we extend the proof of Theorem \ref{theorem-sigma} (i) to general symmetric L\'evy processes $(L_t)_{t\in\R}$. We need the following preliminary result.

\begin{lem}\label{PhiQGro}
Let $q\geq 1$ and $a\in(0, 1].$ The function 
\[\xi(y)=\int_{-a}^a |yx|^2\mathds 1_{\{|yx|\leq 1\}}+|yx|^q\mathds 1_{\{|yx|>1\}}\nu( d x)\]
satisfies $|\xi(y)|\leq C(|y|^2\mathds 1_{\{|y\leq 1|\}}+|y|^{\beta'\vee q}\mathds 1_{\{|y>1|\}})$ for any $\beta'>\beta$, where $C$ does not depend on $a.$
\end{lem}
\begin{proof}
Use the decomposition $\xi=\xi_1+\xi_2$ with
\[\xi_1(y)=\int_{-a}^a |yx|^2\mathds 1_{\{|yx|\leq 1\}}\,\nu(dx),\quad\text{and}\quad \xi_2(y)=\int_{-a}^a |yx|^q\mathds 1_{\{|yx|> 1\}}\,\nu(dx).\]
We obtain
\[\xi_1(y)\mathds 1_{\{|y|\leq 1\}}\leq |y|^2\int_{-1}^1 x^2\nu(dx)\mathds 1_{\{|y|\leq 1\}},\]
and $\xi_1(y)\mathds 1_{\{|y|> 1\}}\leq C|y|^{\beta'\vee q}\mathds 1_{\{|y|>1\}}$ follows from \eqref{Phi_0Est}, showing that $\xi_1$ satisfies the estimate given in the lemma.
 For $q>\beta$ we obtain  
\[\xi_2(y)=2|y|^q \mathds 1_{\{|y|>1/a\}}\int_{1/|y|}^a |x|^q \nu( d x)\leq C|y|^q \mathds 1_{\{|y|\geq 1\}}.\]
If $q\leq \beta$ we have for any $\beta'>\beta$ 
\[\xi_2(y)\leq2|y|^{\beta'} \mathds 1_{\{|y|>1/a\}}\int_{1/|y|}^{a} |x|^{\beta'} \nu( d x)\leq C |y|^{\beta'}\mathds 1_{\{|y|\geq 1\}},\]
which completes the proof.
\end{proof}
 

Now, given a general symmetric L\'evy process $(L_t)_{t\in\R},$ consider for $a>0$ the compound Poisson process $(L^\lj_t)_{t\in\R}$ defined by
\[ L_t^\lj-L_s^\lj= \sum_{s<u\leq t} \Delta L_{u}\mathds 1_{\{|\Delta L_{u}|> a\}},\quad L^\lj_0=0.\]
Moreover, let $(L^\sj_t)_{t\in\R}$ denote the L\'evy process $(L_t-L_t^\lj)_{t\in\R}$. The key result of this section is the following approximation lemma. Intuitively, the lemma shows that replacing $(L_t)_{t\in\R}$ by $(L_t^\lj)_{t\in\R}$ in the definition of $X$ has a negligible effect for $a\to 0.$
\begin{lem}\label{SmaJumCon} It holds that
\begin{align*}
\lim_{a\to 0}\limsup_{n\to\infty}\bigg\| n^{\alpha p} \sum_{i=k}^{[nt_\infty]} \bigg| \int_{-\infty}^{i/n} g_{i,n}(s)\sigma_{s-} \diff L_s^\sj\bigg|^p \bigg\|_1 = 0.
\end{align*}
\end{lem}

\begin{proof}
We make the decomposition
\[\int_{-\infty}^{i/n} g_{i,n}(s)\sigma_{s-} \diff L_s^\sj=A_{i,n}+B_{i,n},\]
where
\[A_{i,n}=\int_{-\delta}^{i/n} g_{i,n}(s)\sigma_{s-} \diff L_s^\sj \quad \text{and }\quad B_{i,n}=\int_{-\infty}^{-\delta} g_{i,n}(s)\sigma_{s-} \diff L_s^\sj.\]
Lemma \ref{KwaWoyIso} shows that
\begin{align*}
\bigg\|n^{\alpha p}\sum_{i=k}^{[nt_\infty]}| A_{i,n}|^p\bigg\|_{1}
&= n^{-1}\sum_{i=k}^{[nt_\infty]} \bigg\|\int_{-\delta}^{i/n} n^{\alpha+1/p}g_{i,n}(s)\sigma_{s-} \diff L_s^\sj\bigg\|_{p}^p\\
&\leq C n^{-1}\sum_{i=k}^{[nt_\infty]}\E\bigg[\big\| F^{i,n} \big\|^p_{p,L^\sj}\bigg],
\end{align*}
where the process $(F^{i,n}_t)_{t\in\R}$ is defined as $F^{i,n}_t=n^{\alpha+1/p}g_{i,n}(t)\mathds 1_{(-\delta,i/n]}(t)\sigma_{t-}$.
Since the random variable $\sup_{t\in[-\delta,\infty)}|\sigma_t|$ is uniformly bounded (see Remark~\ref{eqreassdfsfd}), we obtain by \eqref{NormVsPhiIneq} and \cite[Eq.(4.23)]{BLP} 
\begin{align*}
\E\big[\| F^{i,n} \|^p_{p,L^\sj}\big]
&  \leq C\| n^{\alpha+1/p}g_{i,n}\mathds 1_{[-\delta,i/n]} \|^p_{p,L^\sj}\\
& \leq C |\Phi_{p, L^\sj} (n^{\alpha+1/p}g_{k,n})|^{p/2}\vee |\Phi_{p, L^\sj} (n^{\alpha+1/p}g_{k,n})|\\
& \leq C \bigg(\int_{|x|\leq a} |x|^p+x^2 \nu( d x)\bigg)^{p/2}\vee \bigg(\int_{|x|\leq a} |x|^p+x^2 \nu( d x)\bigg),
\end{align*}
for all $n\in\N$ and $ i\in\{k,\dots,[nt_\infty]\}$. Since $p>\beta$ by assumption, we conclude that 
\begin{align}\label{A_{i,n}Con}
\limsup_{n\to\infty}\bigg\|n^{\alpha p} \sum_{i=k}^{[nt_\infty]} | A_{i,n}|^p\bigg\|_1 \to 0,\quad\text{as }a\to 0.
\end{align}
Next, we show that for all $a>0$ 
\begin{align}\label{B_{i,n}Con}
\limsup_{n \to \infty}\bigg\|n^{\alpha p} \sum_{i=k}^{[nt_\infty]} |B_{i,n}|^p\bigg\|_{1} =0.
\end{align}
Introducing the processes $(Y^{i,n}_t)_{t\in\R}$ and $(Y_t)_{t\in\R}$ defined as
\begin{align*}
Y^{i,n}_t&=n^{\alpha+1/p}g_{i,n}(t)\sigma_{t-}\mathds 1_{(-\infty,-\delta]}(t),\ \text{and}\\
Y_t&= |g^{(k)}(-t)\sigma_{t-}\mathds 1_{(-\infty,-\delta]}(t)|,
\end{align*}
we obtain by Lemma \ref{KwaWoyIso} that
\begin{align}\label{BEst1}
\bigg\|n^{\alpha p} \sum_{i=k}^{[nt_\infty]} |B_{i,n}|^p\bigg\|_{1} &\leq C n^{-1} \sum_{i=k}^{[nt_\infty]} \E\big[\|  Y^{i,n} \|^p_{p,L^\sj}\big].
\end{align}
Moreover, recalling that $|g^{(k)}|$ is decreasing on $(\delta,\infty)$, an application of Lemma \ref{g_{i,n}Est} shows that
\[\E\big[\|  Y^{i,n}  \|^p_{p,L^\sj}\big]\leq n^{p(\alpha +1/p-k)}\E\big[\| Y \|^p_{p,L^\sj}\big],\]
for all $i\in\{k,\dots,n\}.$
Since $\alpha+1/p-k<0$, equation \eqref{B_{i,n}Con} follows if $\E\big[\|  Y \|^p_{p,L^\sj}\big]<\infty.$
 Applying the estimate (\ref{NormVsPhiIneq}) shows that this is satisfied if $\E\big[ \Phi^{1\vee \frac p 2}_{p,L^\sj}( Y)\big]<\infty,$
which is a consequence of (B1) and Lemma \ref{PhiQGro}, where we used that $p>\beta$.
Now, the result follows from \eqref{A_{i,n}Con} and \eqref{B_{i,n}Con}.
\end{proof}

We can complete the proof of Theorem \ref{theorem-sigma} (i) by combining Lemma \ref{SmaJumCon} with the results of Section \ref{ComPoiCas}. To this end, let
\[
X^\lj_t:=\int_{-\infty}^t (g(t-s)-g_0(-s))\sigma_{s-} \diff L^\lj_s,\quad 
X^\sj_t:=\int_{-\infty}^t (g(t-s)-g_0(-s))\sigma_{s-} \diff L^\sj_s,
\]
and introduce the stopping times
\[T_m^\lj := 
\begin{cases}
T_m & \text{if } |\Delta L_{T_m}|>a,\\
\infty & \text{else}.
\end{cases}
\]
The results of Subsection \ref{ComPoiCas} show that
\[n^{\alpha p}V(X^\lj,p;k)_t^n \skorcon Z_t^\lj:=\sum_{m:T^\lj_m\in (0,t]}|\Delta L_{T^\lj_m}\sigma_{T^\lj_m-}|^p V_m \]
for all $a>0,$ where $V(X^\lj,p;k)_t^n$ denotes the power variation of the process $X^\lj.$ Making the decomposition
\begin{align*}
\big(n^{\alpha p}V(p;k)_t^n\big)^{1/p}
&=\big(n^{\alpha p}V(X^\lj,p;k)_t^n\big)^{1/p}+\bigg(\big(n^{\alpha p}V(p;k)_t^n\big)^{1/p}-\big(n^{\alpha p}V(X^\lj,p;k)_t^n\big)^{1/p}\bigg)\\
&:= U^{n,\lj}_t+U^{n,\sj}_t,
\end{align*}
we have by Minkowski's inequality 
\[\lim_{a\to 0}\limsup_{n\to\infty}\P(\sup_{t\in[0,t_\infty]}|U_t^{n,\sj}|>\eps)\leq \lim_{a\to 0}\limsup_{n\to\infty}\P(n^{\alpha p}V(X^\sj,p;k)^n_{t_\infty}>\eps^p)=0,\]
for all $\eps>0$, 
where we applied Lemma \ref{SmaJumCon}. Since $U^{n,\lj}_t \skorcon Z_t^\lj$ as $n\to\infty,$ and $\sup_{t\in[0,t_\infty]}|Z_t^\lj-Z_t|\to 0$ almost surely, as $a\to 0,$ Theorem \ref{theorem-sigma} (i) follows from \cite[Thm 3.2]{Bill}.\qed

\begin{rem} \label{Bernbloc}
A popular strategy for extending laws of large numbers from a class of processes with constant volatility to a more general model, which includes a stochastic volatility factor, is Bernstein's blocking technique. This technique has for example been applied for Brownian semi-stationary processes and  It\^o semimartingales; see e.g. \cite{BCP11,BGJPS}. Also, we will apply this technique in the next subsection to prove Theorem \ref{theorem-sigma}(ii). However, in the framework of Theorem \ref{theorem-sigma}(i) this approach is not applicable. A crucial step for the blocking technique is showing that the asymptotics of the power variation does not change if we replace the increments $\din X$ by $\sigma_{\frac{i-k}n}\din G$, where the process $G$ is defined as
\[G_t=\int_{-\infty}^t \{g(t-s)-g_0(-s)\}~dL_s,\]
see e.g.\ Section~\ref{stableproof} below. 
It can be shown that this is generally not true in the framework of Theorem \ref{theorem-sigma}(i). 
For instance, when $L$ is a compound Poisson process with L\'evy measure $\nu=\delta_{\{1\}}+ \delta_{\{-1\}}$ and $\sigma=L$, this approximation fails.
\end{rem}

\subsection{Proof of Theorem \ref{theorem-sigma} (ii)}\label{stableproof}

Since $t\mapsto V(p;k)^n_t$ is increasing and the limiting function is continuous, uniform convergence on compact sets in probability  follows if we show
\[n^{-1+p(\alpha + 1/\beta)}V(p;k)^n_t \toop m_p \int_0^t |\sigma_s|^p ds\]
for a fixed $t>0,$ which we will do in the following. 

As already mentioned, we will use Bernstein's blocking technique  to prove Theorem~\ref{theorem-sigma}(ii). A crucial step in the proof is to show that the  asymptotic behavior of the power variation does not change if we replace $\din X$ in \eqref{vn} by $\sigma_{(i-k)/n}\din G$, where the process $(G_t)_{t\geq 0}$ is defined as 
in Remark \ref{Bernbloc}.
Note that assumption (A) ensures that $G$ is well-defined.
Thereafter, we apply the following blocking technique. We divide the interval $[0,t]$ into subblocks of size $1/l$ and freeze $\sigma$ at the beginning of each block. The limiting power variation for the resulting process can then be derived by applying part (ii) of Theorem \ref{maintheorem} on every block. The proof of Theorem \ref{theorem-sigma} (ii) is then completed by letting $l\to\infty.$

The following lemma plays an important role when replacing $\din X$ in \eqref{vn} by $\sigma_{(i-k)/n}\din G$. Here and in the following we denote by $v_\sigma$ the modulus of continuity of $\sigma$ defined  as 
\[v_\sigma(s,\eta)=\sup\{|\sigma_s-\sigma_r|:  r\in[s-\eta,s+\eta]\}.\]
\begin{lem}\label{vineq} Let $(\sigma_t)_{t\in\R}$ be a process with c\`adl\`ag or c\`agl\`ad sample paths that is uniformly bounded on $[-\delta,\infty)$. For any $\alpha,q\in(0,\infty)$ and any (deterministic) sequence $(a_n)_{n\geq 1}$ with $a_n\to 0$ we have
\[\lim_{\eps\to 0}\left[\limsup_{n\to\infty} \left(\frac 1 n\suonetn\|v_{\sigma}(i/n , \eps+a_n)\|^\alpha_q \right)\right]= 0.\]
\end{lem}
\begin{proof}
Since $v_\sigma$ is bounded and $x\mapsto x^\alpha$ is locally Lipschitz for $\alpha>1,$ we may assume w.l.o.g. that $\alpha\leq 1$ and $q\geq 1$. For $\kappa>0$ we use the decomposition
$\sigma=\sigma^{\kappas}+\sigma^{\kappal},$
where
\[\sigma_s^{\kappal}=\sum_{-\delta< u\leq s}\Delta\sigma_u\mathds 1_{\{|\Delta\sigma_u|\geq \kappa\}},\]
and $\sigma_s^{\kappas}=\sigma_s-\sigma_s^{\kappal}.$ Eventhough $\sigma$ is uniformly bounded on $[-\delta,\infty)$, $\sigma^{\kappal}$ and $\sigma^{\kappas}$ might not be. For this reason we introduce the sets 
\begin{align*}
\Omega_m:=\big\{
	&\omega\,:\, |\sigma_s^{\kappas}(\omega)|+|\sigma_s^{\kappal}(\omega)|\leq m \text{ for all }s\in[-\delta,t+\delta], \\
	&\text{ and $\sigma^{\kappal}(\omega)$ has less than $m$ jumps in $[-\delta,t+\delta]$}\big\}.
\end{align*}
Note that $\Omega_m\uparrow \Omega$, as $m\to\infty.$
By triangular inequality we have $v_\sigma(s,\eta)\leq v_{\sigma^{\kappas}}(s,\eta)\mathds 1_{\Omega_m}+v_{\sigma^{\kappal}}(s,\eta)\mathds 1_{\Omega_m}+C\mathds 1_{\Omega_m^{\mathrm c}}$ for all $s\in [0,t],\eta<\delta$ and $m\geq 1$. Since $\P(\Omega_m^{\mathrm c})\to 0$ as $m\to\infty$, we can choose $m$ sufficiently large such that 
\begin{align}\label{VEst}
&\frac 1 n\suonetn\|v_\sigma(i/n , \eps+a_n)\|^\alpha_q \nonumber\\
&\hspace{2em} \leq \frac 1 n\suonetn\|v_{\sigma^{\kappas}}(i/n , \eps+a_n)\mathds 1_{\Omega_m}\|^\alpha_q +\frac 1 n\suonetn\|v_{\sigma^{\kappal}}(i/n , \eps+a_n)\mathds 1_{\Omega_m}\|^\alpha_q +\kappa,
\end{align}
for all $n\in\N$ and $\eps>0.$
We show that
\begin{align}\label{kappacontin}
\limsup_{\eps\to 0}\,\limsup_{n\to\infty}\left(\frac 1 n\suonetn\|v_{\sigma^{\kappas}}(i/n , \eps+a_n)\mathds 1_{\Omega_m}\|^\alpha_q \right)\leq 2\kappa^\alpha.
\end{align}
In order to do so, we assume the existence of sequences $(\eps_l),(n_l),(i_l)$ with $\eps_l\to 0,$ $n_l\to\infty$ and $i_l\in\{1,...,[tn_l]\}$ such that
\begin{align}\label{contra1}
 \|v_{\sigma^{\kappas}}(i_l/n_l, \eps_l+a_{n_l})\mathds 1_{\Omega_m}\|^\alpha_q >2\kappa^\alpha 
 \end{align}
for all $l,$ and derive a contradiction. Since $(i_l/n_l)_{l\geq 1}$ is a bounded sequence we may assume that $i_l/n_l$ converges to some $s_0\in[0,t]$ by considering a suitable subsequence $(l_k)_{k\geq 1}.$ 
For all $\omega\in\Omega_m$ it holds that $\lim_{\gamma\to 0} v_{\sigma^{\kappas}}(s_0 , \gamma)=|\Delta \sigma^{\kappas}_{s_0}|\leq \kappa$. Therefore, by the dominated convergence theorem, we can find a $\gamma>0$ such that $\|v_{\sigma^{\kappas}}(s_0 , \gamma)\mathds 1_{\Omega_m}\|^\alpha_q\leq 2\kappa^\alpha$.
This is a contradiction to \eqref{contra1}, since for sufficiently large $l$ we have $[i_l/n_l-\eps_l-|a_{n_l}|,i_l/n_l+\eps_l+|a_{n_l}|]\subset [s_0-\gamma,s_0+\gamma].$
This completes the proof of (\ref{kappacontin}).

Next, we show that 
\begin{align}\label{VLarJum}
\lim_{\eps\to 0}\limsup_{n\to \infty}\left(\frac 1 n\suonetn\|v_{\sigma^{\kappal}}(i/n , \eps+a_n)\mathds 1_{\Omega_m}\|^\alpha_q\right)=0.
\end{align}
Recalling that $q/\alpha\geq 1$, an application of Jensen's inequality yields
\begin{align*}
\frac 1 n \suonetn \|v_{\sigma^{\kappal}}(i/n , \eps+a_n) \mathds 1_{\Omega_{m}}\|^\alpha_q
&\leq \left(t^{q/\alpha-1}\frac 1 n \suonetn\|v_{\sigma^{\kappal}}(i/n , \eps+a_n) \mathds 1_{\Omega_{m}}\|_q^q\right)^{\alpha/q}\nonumber\\
&=\bigg\| t^{q/\alpha-1}\frac 1 n \suonetn \big(v_{\sigma^{\kappal}}(i/n , \eps+a_n) \mathds 1_{\Omega_{m}}\big)^q\bigg\|_1^{\alpha/q},
\end{align*}
for all $n\in\N,\ \eps>0.$ Now, \eqref{VLarJum} follows from the estimate
\[\frac 1 n\suonetn\big(v_{\sigma^{\kappal}}(i/n , \eps+a_n)\mathds 1_{\Omega_m}\big)^q\leq \sup_{s\in[-\delta,t+\delta]}|\Delta\sigma^\kappal_s|^q N\mathds 1_{\Omega_m}2(\eps+a_n)\leq C m^{q+1}(\eps+a_n),\]
for all $n\in\N$.
Here $N=N(\omega)$ denotes the number of jumps of $\sigma^\kappal$ in $[-\delta,t+\delta].$
Using \eqref{kappacontin} and \eqref{VLarJum}, the lemma now follows from \eqref{VEst} by letting $\kappa\to 0.$
\end{proof}

The proof of \ref{theorem-sigma} (ii) heavily relies on the estimate given in Lemma \ref{WeaLBetIso}. This lemma assumes the role that It\^o's isometry typically plays for the blocking technique when the involved stochastic integral is driven by a Brownian motion. In order to apply Lemma \ref{WeaLBetIso}, the following estimates will be crucial.

\begin{lem}\label{StableLIntegrandEstimate} Suppose that assumption (B2) holds, and assume that $\alpha+1/\beta<k$. For $\eps>0$ with $\eps\leq \delta$ there is a constant $C>0$ such that
\begin{align}
\E\bigg[\int_{\frac i n-\eps}^{\frac in}|g_{i,n}(s)\sigma_{s-}|^\beta \diff s\bigg]+\int_{\frac i n-\eps}^{\frac i n}|g_{i,n}(s)|^\beta \diff s 
&\leq C n^{-\alpha\beta-1},\quad \text{and}\\
\E\bigg[\int_{-\infty}^{\frac i n-\eps}|g_{i,n}(s)\sigma_{s-}|^\beta \diff s\bigg]+\int_{-\infty}^{\frac i n-\eps}|g_{i,n}(s)|^\beta \diff s
&\leq C n^{-k\beta },
\end{align}
for all $i\in\{k,\dots,n\}.$
\end{lem}
\begin{proof} By Lemma \ref{g_{i,n}Est} we have that
\begin{align*}
&|g_{i,n}(s)|^\beta\mathds 1_{[i/n-\eps,i/n]}(s)\\
&\eqspace\leq C\big((i/n-s)^{\alpha\beta}\mathds 1_{[(i-k)/n,i/n]}(s)+n^{-k\beta }((i-k)/n-s)^{(\alpha -k)\beta}\mathds 1_{[i/n-\eps,(i-k)/n]}(s)\big).
\end{align*}
Recalling that $\sigma$ is bounded on $[-\delta,\infty),$ the first inequality follows by calculating the integral of the right hand side.
The second inequality is a direct consequence of Lemma \ref{g_{i,n}Est} and assumptions (A) and (B2).
  \end{proof}
A crucial step in the proof of Theorem \ref{theorem-sigma} (ii) is showing that
\begin{align}
 n^{-1+p(\alpha+1/\beta)}\sum_{i=k}^{[nt]}\|\din X -\sigma_{(i-k)/n}\din G\|_p^p\to 0\label{sum},
\end{align}
as $n\to\infty,$ where the process $G$ is defined in Remark \ref{Bernbloc}.
We fix some $\eps>0$ and make the decomposition
\[\din X -\sigma_{(i-k)/n}\din G = A_i^{n,\eps} +B_i^{n,\eps} + C_i^{n,\eps},\]
where
\begin{align*}
A_i^{n,\eps}	&= \int_{i/n-\eps}^{i/n} g_{i,n}(s)(\sigma_{s-}-\sigma_{i/n-\eps}) \diff L_s,\\
B_i^{n,\eps}	&= (\sigma_{i/n-\eps}-\sigma_{(i-k)/n}) \int_{i/n-\eps}^{i/n} g_{i,n}(s) \diff L_s,\\
C_i^{n,\eps}	&=\int_{-\infty}^{i/n-\eps} g_{i,n}(s)\sigma_{s-} \diff L_s - \sigma_{(i-k)/n}\int_{-\infty}^{i/n-\eps} g_{i,n}(s)\diff L_s.
\end{align*}
We deduce \eqref{sum} by showing that
\[\lim_{\eps\to 0}\limsup_{n\to\infty}\bigg( n^{-1+p(\alpha+1/\beta)}\sum_{i=k}^{[nt]} \|A_{i}^{n,\eps}\|^p_{p}\bigg)=0,\]
and the same for $B_i^{n,\eps}$ and $ C_i^{n,\eps}$, respectively.
For $A_i^{n,\eps}$ we obtain by Lemma \ref{WeaLBetIso}
\begin{align*}
&n^{-1+p(\alpha+1/\beta)}\sum_{i=k}^{[nt]} \|A_{i}^{n,\eps}\|^p_{p}\\
&\eqspace\leq C n^{-1+p(\alpha+1/\beta)} \sum_{i=k}^{[nt]} \bigg\{\E\bigg[\int_{i/n-\eps}^{i/n} \big|g_{i,n}(s)(\sigma_{s-}-\sigma_{i/n-\eps})\big|^\beta \diff s\bigg]\bigg\}^{p/\beta}\\
&\eqspace\leq C n^{-1+p(\alpha+1/\beta)} \sum_{i=k}^{[nt]} \|v_\sigma(i/n,\eps+1/n)\|^p_\beta\left( \int_{i/n-\eps}^{i/n} |g_{i,n}(s)|^\beta  \diff s\right)^{p/\beta}.\\
\end{align*}
By Lemma \ref{vineq} and Lemma \ref{StableLIntegrandEstimate} we conclude that
\begin{align}\label{ACon}
\lim_{\eps\to 0}\limsup_{n\to\infty} \bigg(n^{-1+p(\alpha+1/\beta)}\suonetn \|A_{i}^{n,\eps}\|^p_{p}\bigg)= 0.
\end{align}
For $B_{i}^{n,\eps}$ we apply H\"older's inequality with $p'$ and $q'$ satisfying $1/p'+1/q'=1$ and $pq'<\beta$, which is possible due to our assumption $p<\beta.$ This yields
\begin{align*}
&n^{-1+p(\alpha+1/\beta)}  \sum_{i=k}^{[nt]} \|B_{i}^{n,\eps}\|^p_{p}\\
&\eqspace\leq n^{-1+p(\alpha+1/\beta)}  \sum_{i=k}^{[nt]} \|(\sigma_{i/n-\eps}-\sigma_{(i-k)/n})\|^p_{pp'}\bigg\|\int_{i/n-\eps}^{i/n} g_{i,n}(s)\diff L_s\bigg\|^p_{pq'},\\
&\eqspace\leq C n^{-1} \sum_{i=k}^{[nt]} \big\|v_\sigma(i/n,\eps+k/n)\big\|^p_{pp'}.
\end{align*}
Here we have used that, as a consequence of Lemma \ref{WeaLBetIne} and Lemma \ref{StableLIntegrandEstimate}, whenever $pq'<\beta$ there exists a $C>0$ such that $\|n^{\alpha+1/\beta}\int_{i/n-\eps}^{i/n} g_{i,n}(s)\diff L_s\|_{pq'}<C$ for all $n\in\N,$ $i\in\{k,...,[nt]\}$.
Thus, by Lemma \ref{vineq}
\begin{align}\label{BCon}
\lim_{\eps\to 0}\limsup_{n\to\infty}\bigg( n^{-1+p(\alpha+1/\beta)}\suonetn \|B_{i}^{n,\eps}\|^p_{p}\bigg)= 0.
\end{align}
Moreover, by Lemma \ref{WeaLBetIso} and Lemma \ref{StableLIntegrandEstimate} it follows that for all $\eps>0$
\begin{align}
 \limsup_{n\to\infty}\bigg(n^{-1+p(\alpha+1/\beta)}\suonetn \|C_{i}^{n,\eps}\|^p_{p}\bigg)
&\leq C\limsup_{n\to\infty}(n^{p(\alpha+1/\beta-k)})=0,
\end{align} 
which together with \eqref{ACon} and \eqref{BCon} completes the proof of \eqref{sum}.

By Minkowski's inequality for $p\geq 1$ and subadditivity for $p<1$, it is now sufficient to show that
\begin{align}\label{SubLimSta}
n^{-1+p(\alpha+1/\beta)}\suonetn |\sigma_{(i-1/n)}\din G|^p\toop m_p\int_0^t|\sigma_s|^p \diff s,
\end{align}
in order to prove Theorem \ref{theorem-sigma} (ii).

Intuitively, replacing $|\din X|$ by $|\sigma_{(i-k)/n}\din G|$ corresponds to freezing the process $(\sigma_t)_{t\in\R}$ over blocks of length $1/n.$ For the prove of \eqref{SubLimSta} we freeze $\sigma$ now over small blocks with block size $1/l$ that does not depend on $n$. This will allow us to apply Theorem \ref{maintheorem} (ii) on every block. Thereafter, \eqref{SubLimSta} follows by letting $l\to\infty.$ For $l>0$ we decompose 
\begin{align*}
&n^{-1+p(\alpha+1/\beta)}\sum_{i=k}^{[nt]} |\sigma_{(i-k)/n}\din G|^p- m_p\int_0^t|\sigma_s|^p \diff s \\
&\hspace{+3em}=n^{-1+p(\alpha+1/\beta)}\bigg(\sum_{i=k}^{[nt]} |\din G|^p\big(|\sigma_{(i-k)/n}|^p-|\sigma_{(j_{l,i}-1)/l}|^p\big)\bigg)\\
&\hspace{4em}+\bigg(n^{-1+p(\alpha+1/\beta)}\sum_{j=1}^{[tl]+1}|\sigma_{(j-1)/l}|^p\bigg(\sum_{i\in I_l(j)}|\din G|^p-m_pl^{-1}\bigg)\bigg)\\
&\hspace{4em}+ \bigg(m_pl^{-1}\sum_{j=1}^{[tl]}|\sigma_{(j-1)/l}|^p-m_p\int_0^t|\sigma_s|^p ds\bigg)\\
&\hspace{3em}=D_{n,l}+E_{n,l} +F_{l}.
\end{align*}
Here, $j_{l,i}$ denotes the index $j\in\{1,...,[tl]+1\}$ such that $(i-k)/n\in ((j-1)/l,j/l]$ and $I_l(j)$ is the set of indices $i$ such that $(i-k)/n\in((j-1)/l,j/l].$
We show that 
\[\lim_{l\to\infty}\limsup_{n\to\infty}\P(|D_{n,l}+E_{n,l} +F_{l}|>\eps)=0\] for any $\eps>0$. 
Note that $F_l\toas 0$ as $l\to\infty$, since the Lebesgue integral of any c\`adl\`ag function exists. 
For every $l\in\N$ we have $\limsup_{n\to\infty}\P(|E_{n,l}|>\eps)=0$ by 
Theorem~\ref{maintheorem}(ii).
For $\lim_{l\to\infty}\limsup_{n\to\infty}\P(|D_{n,l}|>\eps)=0$ we argue as follows. 
Choose some $p'>1$ such that $pp'<\beta$ and let $q'$ be such that $1/p'+1/q'=1.$ We find 
\begin{align*}
\|D_{n,l}\|_1&=\bigg\|n^{-1+p(\alpha+1/\beta)}\bigg(\sum_{i=k}^{[nt]}  |\din G|^p(|\sigma_{(i-k)/n}|^p-|\sigma_{(j_{l,n,i}-1)/l}|^p)\bigg)\bigg\|_1\\
&\leq n^{-1}\sum_{i=k}^{[nt]} \||n^{\alpha+1/\beta}\din G|^p\|_{p'}\||\sigma_{(i-k)/n}|^p-|\sigma_{(j_{l,n,i}-1)/l}|^p\|_{q'}\\
&\leq\bigg(n^{-1}\sum_{i=k}^{[nt]} \|n^{\alpha+1/\beta}\din G\|^{2/p}_{pp'}\bigg)^{1/2}\bigg(n^{-1}\sum_{i=k}^{[nt]} \||\sigma_{(i-k)/n}|^p-|\sigma_{(j_{l,n,i}-1)/l}|^p\|^2_{q'}\bigg)^{1/2}.
\end{align*}
The first factor is bounded by Lemmas \ref{WeaLBetIso} and \ref{StableLIntegrandEstimate}.
For the second factor we can apply Lemma \ref{vineq}, since the process $ (|\sigma_t|^p)_{t\in\R}$ is c\`adl\`ag and bounded on $[-\delta,\infty),$ and conclude that $\lim_{l\to\infty}\limsup_{n\to\infty}\|D_{n,l}\|_1=0.$ This completes the proof of \eqref{SubLimSta}, and hence of Theorem \ref{theorem-sigma} (ii).\qed


\subsection{Proof of Theorem \ref{theorem-sigma} (iii)}


 For the proof of Theorem \ref{theorem-sigma} (iii) we show that under the conditions of the theorem the process $X$ admits a modification with $k$-times differentiable sample paths with $k$-th derivative $F$ as defined in the theorem. Then the result follows by an application of the following stochastic Fubini theorem. We remark that a Fubini theorem for L\'evy integrals of deterministic integrands was derived by similar means in \cite[Theorem 3.1]{BB11}.

\begin{lem}\label{Fub}
Let $f:\R\times\R\times\Omega\to\R$ be a random field that is measurable with respect to the product $\sigma$-algebra $\mathcal B(\R)\otimes \Pi$, where $\Pi$ denotes the $(\mathcal F_t)_{t\in\R}$-predictable $\sigma$-algebra on $\R\times\Omega$. That is, $\Pi$ is the $\sigma$-algebra generated by all sets $A\times (s,t]$, where $s<t$ and $A\in \mathcal F_s$. Let $(L_t)_{t\in\R}$ be a symmetric L\'evy process that has finite first moment. Assume that we have
\begin{align}\label{IntConFub1}
\E\bigg[\int_{\R} \|f(u,\cdot)\|_{1,L}\diff u\bigg]<\infty.
\end{align}
Then, we obtain
\[\int_{\R} \bigg(\int_\R f(u,s) \diff u \bigg)\diff L_s= \int_{\R}\bigg(\int_\R f(u,s)  \diff L_s\bigg)\diff u\quad \text{almost surely},\]
and all the integrals are well-defined.
\end{lem}
The proof of this result relies on the following lemma.

\begin{lem}\label{NorIntCha}
Let $L$ be a L\'evy process that has $q$-th moment for some $q\geq 1$. There exists a $C>0$ such that the following holds. For all random fields $f:\R\times\R\times\Omega\to\R$ that are measurable with respect to the product $\sigma$-algebra $\mathcal B(\R)\otimes \Pi$, where $\Pi$ denotes the $(\mathcal F_t)_{t\in\R}$-predictable $\sigma$-algebra on $\R\times \Omega$,  we have that
\begin{align}\label{NorIntIne}
\bigg\|\int_\R |f(u,\cdot)|\diff u\bigg\|_{q,L}\leq C\int_\R\| f(u,\cdot)\|_{q,L}\diff u \quad \text{almost surely.}
\end{align}
\end{lem}
\begin{proof}
Let us first remark that the stochastic process $(\| f(u,\cdot)\|_{q,L})_{u\in \R}$ indeed admits a measurable modification by the following argument. For any $g\in \ldl{q}{L}$, the mapping $\Gamma_g:\R\to\R,\ u\mapsto\E[\|f(u,\cdot)-g\|_{q,L}\wedge 1]$
is measurable by product measurebility of $f$. Therefore, the mapping $\R\to \ldl{q}{L}, u\mapsto f(u,\cdot)$ is measurable, where we equip $\ldl{q}{L}$ with the (semi)metric induced by $\E[\|\cdot\|_{q,L}\wedge 1]$.
Since the mapping $\ldl{q}{L}\to L^0(\Omega), g\mapsto \|g\|_{q,L}$ is continuous, the existence of a measurable modification of $(\| f(u,\cdot)\|_{q,L})_{u\in \R}$ follows from \cite[Theorem 3]{C72}.

The proof relies on the monotone class theorem, cf. \cite[Theorem 6.1.3]{Dur}. We fix a compact set $K\subset \R$ and denote by $\mathcal L^0_K$ the linear space of all random fields $f:K\times K\times\Omega\to\R$ that are measurable with respect to $\mathcal B(K)\otimes \Pi$.
 Denote by $\mathcal R_K$ the set of all random fields in $\mathcal L^0_K$ of the form
\[f(u,s)=\sum_{j=1}^l F^j_s \mathds 1_{A_j}(u),\]
where $l\in\N$, and $F^1,\dots,F^l\in \ldl{q}{L}$ are bounded, and $A_1,\dots A_l$ are disjoint sets in $\mathcal B(K).$ Then, \eqref{NorIntIne} holds for all $f\in\mathcal R_K$ with $C$ as in \eqref{triangleine}. 
 Let 
\begin{align*}
\mathcal H_K:=\bigg\{ f\in \mathcal L_K^0\,:\,&\text{ There is a sequence $(f_n)_{n\in\N}\subset\mathcal R_K$ such that }\\
&\int_K \|f(u,\cdot)-f_n(u,\cdot)\|_{q,L}\diff u\to 0,\text{ a.s., and }\\
&\bigg\|\int_K |f(u,\cdot)-f_n(u,\cdot)|\diff u\bigg\|_{q,L}\to 0,\text{ a.s.}\bigg\}. 
\end{align*}
Note that $\mathcal H_K$ is a linear space containing $\mathcal R_K$ that is closed under bounded monotone convergence. Therefore, the monotone class theorem yields that $\mathcal H_K$ contains all bounded functions in $\mathcal L^0_K$. Recalling the equivalence of $\|\cdot\|_{q,L}$ and the (random) norm $\|\cdot\|'_{q,L},$ which has been introduced in Section \ref{sec2}, it is easy to see that all $f\in \mathcal H_K$ satisfy \eqref{NorIntIne}, with the constant $C$ not depending on $K$. Now, the lemma follows for general $f$ by an approximation argument.
\end{proof}

With this result at hand we can now prove Lemma \ref{Fub}. 

\begin{proof}[Proof of Lemma \ref{Fub}]
Let us first remark that the stochastic process $u\mapsto \int_\R f(u,s)  \diff L_s$ admits a measurable modification by measurability of the map $\R\to\ldl{1}{L},\ u\mapsto  f(u,\cdot)$ and continuity of the integral mapping $\ldl{1}{L}\to L^1(\Omega),\ f(u,\cdot)\mapsto \int_\R f(u,s)\diff L_s$, which follows from Lemma \ref{KwaWoyIso}. Then, the existence of a measurable modification follows from \cite[Theorem 3]{C72}.

The result obviously holds for $f$ of the form
\begin{align}\label{SimFunFub}
f(u,s,\omega)=\sum_{i=1}^k \alpha_i\mathds 1_{A_i}(u)\mathds 1_{B_i}(s,\omega),
\end{align}
where $\alpha_i\in\R,$ $A_i\in\mathcal B(\R)$ and $B_i\in\Pi$ for $i=1,\dots, k$.
By an application of the monotone class theorem we see that for every $\mathcal B(\R)\otimes\Pi$ measurable $f$ satisfying \eqref{IntConFub1} we can find a sequence $f_n$ of the form \eqref{SimFunFub} such that $\E[\int_\R\|f_n(u,s)-f(u,s)\|_{L,1}\diff u]\to 0.$
Letting 
\[Y:=\int_{\R}\bigg(\int_\R f(u,s)  \diff L_s\bigg)\diff u,\quad \text{and}\quad Y_n:=\bigg[\int_{\R}\bigg(\int_\R f_n(u,s) \diff L_s\bigg)\diff u\bigg],\]
we have by \eqref{IntConFub1} and Lemma \ref{KwaWoyIso}
\[\E[|Y|]\leq C\int_{\R} \E[\|f(u,\cdot)\|_{1,L}]\diff u<\infty.\]
Thus, we deduce that
\[\E[|Y-Y_n|]\leq C\E\bigg[\int_\R \|f(u,\cdot)-f_n(u,\cdot)\|_{1,L}\diff u\bigg]\to 0,\quad\text{as $n\to\infty$.}\]
By applying Lemma \ref{NorIntCha} we obtain that 
\[Z:=\int_{\R}\bigg(\int_\R f(u,s) \diff u\bigg) \diff L_s\quad\text{and}\quad Z_n=\int_{\R}\bigg(\int_\R f_n(u,s)\diff u\bigg) \diff L_s\]
are in $L^1(\Omega),$ and that $\E[|Z_n-Z|]\to 0,$ as $n\to\infty.$
Now the result follows since $Y_n=Z_n$ almost surely.
\end{proof}

In order to apply Fubini's theorem in the proof of Theorem \ref{theorem-sigma} (iii) we need the following lemma.

\begin{lem}\label{FubCon}
Suppose that assumption (B1) holds and let $\alpha>k-1/( \beta\vee 1)$. For any $t>0$ the random field  $f_{t}(u,s):=g^{(k)}(u-s)\sigma_{s-}\mathds 1_{[0,t]}(u)\mathds 1_{(-\infty,u)}(s)$ satisfies the conditions of Lemma \ref{Fub}.
\end{lem}
\begin{proof}
The measurability with respect to $\mathcal B(\R)\otimes \Pi$ is obvious.
By \eqref{NormVsPhiIneq} for the function $f_{t}$ to satisfy \eqref{IntConFub1} it is sufficient to show that
 \begin{align*}
 \int_0^t \E[\Phi_{1,L}(f_t(u,\cdot))]\diff u
&= \int_0^t \E\big[\Phi_{1,L}\big(f_t(u,\cdot)\mathds 1_{(-\infty, -\delta]}\big)\diff u+\int_0^t\E\big[\Phi_{1,L}\big(f_t(u,\cdot)\mathds 1_{ (-\delta,t]}\big)\big]\diff u\\
&:= I_1+I_2<\infty.
\end{align*}
We first show that $I_1<\infty.$ Since $|g^{(k)}|$ is decreasing on $(\delta,\infty)$, we have for all $u\in[0,t]$
\begin{align*}
\Phi_{1,L}\big(f_t(u,\cdot)\mathds 1_{(-\infty,-\delta]}\big)&\leq \Phi_{1,L}\big(f_t(0,\cdot)\mathds 1_{(-\infty,-\delta]}\big)\\
&=\int_{-\infty}^{-\delta}\int_\R |f_t(0,s)x|^2 \mathds 1_{\{|f_t(0,s)x|\leq 1\}}\ \nu( d x)\diff s\\
&\hspace{1em}+ \int_{-\infty}^{-\delta}\int_\R |f_t(0,s)x| \mathds 1_{\{|f_t(0,s)x|> 1\}}\ \nu( d x)\diff s\\
&=J_1+J_2.
\end{align*}
By \eqref{Phi_0Est} we obtain for any $\beta'>\beta$
\begin{align}\label{J_1}
J_1\leq C \int_{-\infty}^{-\delta} |f_t(0,s)|^\theta\mathds 1_{\{|f_t(0,s)|\leq 1\}}+|f_t(0,s)|^{\beta'}\mathds 1_{\{|f_t(0,s)|> 1\}}\diff s.\end{align}
For $\beta'$ close enough to $\beta$, the right hand side of \eqref{J_1} is in $L^1(\Omega)$ by assumption (B1).
For $J_2$ we observe for all $\beta'>\beta$ with $\beta'\geq 1$ that
\begin{align}\label{J_2Est}
J_2 &\leq \int_{-\infty}^{-\delta}\int_{-1}^1 |f_t(0,s)x|^{\beta'} \mathds 1_{\{|f_t(0,s)x|> 1\}}\ \nu( d x)\diff s +2
 \int_{-\infty}^{-\delta}|f_t(0,s)|\diff s \int_{1}^\infty |x|\ \nu( d x)\nonumber\\
 &\leq C \int_{-\infty}^{-\delta}|f_t(0,s)|^{\beta'}\mathds 1_{\{|f_t(0,s)|> 1\}}+|f_t(0,s)|\diff s,
\end{align}
which is in $L^1(\Omega)$ by assumption (B1) for sufficiently small $\beta'>\beta$. Thus, we conclude that $I_1<\infty.$

Now we show that $I_2<\infty.$ By boundedness of $\sigma$ on the interval $[-\delta,t],$ it holds for all $u\in[0,t]$ that
\begin{align*}
&\Phi_{1,L}\big(f_t(u,\cdot)\mathds 1_{ (-\delta,t]}\big)\\
&\leq C \int_{-\delta}^t \int_\R |g^{(k)}(t-s)x|\mathds 1_{\{|g^{(k)}(t-s)x|>1\}}+|g^{(k)}(t-s)x|^2\mathds 1_{\{|g^{(k)}(t-s)x|\leq 1\}}\nu( d x)\diff s.
\end{align*}
We recall that $|g^{(k)}(s)|\leq Cs^{\alpha-k}$ for $s\in(0, \delta].$ Thus, choosing $\beta'>\beta$ such that $(\alpha-k)\beta'>-1$, we have by \eqref{Phi_0Est}
\[\int_{-\delta}^t \int_\R |g^{(k)}(t-s)x|^2\mathds 1_{\{|g^{(k)}(t-s)x|\leq 1\}}\nu( d x)\diff s\leq C +\int_{-\delta}^t  |g^{(k)}(t-s)|^{\beta'}\diff s<\infty.\]
Moreover, by arguing as in \eqref{J_2Est}
\begin{align*}
 &\int_{-\delta}^t \int_\R |g^{(k)}(t-s)x|\mathds 1_{\{|g^{(k)}(t-s)x|>1\}}\nu( d x)\diff s\\
 &\leq C\int_{-\delta}^t  |g^{(k)}(t-s)|^{\beta'}\mathds 1_{\{|g^{(k)}(t-s)|>1\}}+ |g^{(k)}(t-s)|\diff s<\infty,
\end{align*}
where we used that $\alpha-k>-1.$
Therefore, we conclude that $I_2<\infty$, which completes the proof.
\end{proof}

In the proof of Theorem \ref{theorem-sigma} (iii) we will again separate the large jumps and small jumps of the L\'evy process $L.$ For the small jump part the result given in Lemma \ref{FubCon} can be strengthened as follows. Note that for this result it is in fact necessary to cut off the large jumps of $L$, since $L$ might not have $p$-th moment.

\begin{lem}\label{SmaJumLpEst}
Let $t>0$ be fixed, and let $p\geq 1$. Suppose, (B1) is satisfied and suppose that $\alpha>k-1/( \beta\vee p)$. For any $a\in(0,1]$, the random field $f_t(u,s)=g^{(k)}(u-s)\sigma_{s-}\mathds 1_{[0,t]}(u)\mathds 1_{(-\infty,u)}(s)$ satisfies
\[\int_0^t\E[\|f_t(u,\cdot)\|^p_{p,L^\sj}]\diff u<\infty,\]
where process $(L^\sj_s)_{s\in\R}$ is defined in Lemma \ref{SmaJumCon}.
\end{lem}
\begin{proof}
We decompose
\begin{align*}
\int_0^t\E[\|f_t(u,\cdot)\|^p_{p,L^\sj}]\diff u&\leq C\int_0^t\E[\|f^{(1)}_t(u,\cdot)\|^p_{p,L^\sj}]\diff u+C\int_0^t\E[\|f^{(2)}_t(u,\cdot)\|^p_{p,L^\sj}]\diff u\\
&= I_1+I_2,
\end{align*}
where 
\[f^{(1)}_t(u,s)=g^{(k)}(u-s)\sigma_{s-}\mathds 1_{[0,t]}(u)\mathds 1_{(-\delta,u)}(s),\quad f^{(2)}_t(u,s)=g^{(k)}(u-s)\sigma_{s-}\mathds 1_{[0,t]}(u)\mathds 1_{(-\infty,-\delta]}(s).\]
For $I_1$ we use that $\sigma$ is bounded on $[-\delta,\infty).$ Thus, denoting $e_t(u,s)=g^{(k)}(u-s)\mathds 1_{[0,t]}(u)\mathds 1_{(-\delta,u)}(s),$ we have that
\begin{align*}
I_1&\leq C \int_0^t \| e_t(u,\cdot)\|^p_{p,L^\sj} \diff u\\
&\leq C \int_0^t \Phi_{p,L^\sj}( e_t(u,\cdot))+\Phi^{\frac p 2}_{p,L^\sj}( e_t(u,\cdot)) \diff u\\
&\leq Ct \big(\Phi_{p,L^\sj}( e_t(t,\cdot)) +\Phi^{\frac p 2}_{p,L^\sj}( e_t(t,\cdot))\big),
\end{align*}
where the second inequality follows from \eqref{NormVsPhiIneq},  the third inequality holds since $|e_t(u,s)|\leq |e_t(t,s+t-u)|$, and since $\Phi_{p,L^\sj}(f)$ is invariant under shifting the argument of the function $f$. We obtain by virtue of Lemma \ref{PhiQGro}
\begin{align*}
\Phi_{p,L^\sj}( e_t(t,\cdot))\leq C\int_{-\delta}^t |g^{(k)}(t-s)|^2\mathds 1_{\{|g^{(k)}(t-s)|\leq 1\}}+|g^{(k)}(t-s)|^{\beta'\vee p}\mathds 1_{\{|g^{(k)}(t-s)|> 1\}}\diff s,
\end{align*}
where we choose $\beta'>\beta$ such that $(\alpha-k)(\beta'\vee p)>-1$.
Now, recalling that $|g^{(k)}(t)|\leq Ct^{\alpha-k}$ for all $t\in(0,\delta)$, we see that
\[\int_{-\delta}^t |g^{(k)}(t-s)|^{\beta'\vee p}\mathds 1_{\{|g^{(k)}(t-s)|> 1\}}\diff s<\infty.\]
Consequently, $\Phi_{p,L^\sj}( e_t(t,\cdot))$ is finite which implies $I_1<\infty.$

For $I_2$ we obtain that
 $\E[\|f^{(2)}_t(u,\cdot)\|^p_{p,L^\sj}]\leq \E[\|f^{(2)}_t(0,\cdot)\|^p_{p,L^\sj}]$ for all $u\in [0,t],$ because $|g^{(k)}|$ is decreasing on $(\delta,\infty).$ By \eqref{NormVsPhiIneq} it holds that 
\[\E[\|f^{(2)}_t(0,\cdot)\|^p_{p,L^\sj}]\leq \E\bigg[\Phi_{p,L^\sj}\big(f^{(2)}_t(0,\cdot)\big)\bigg]+\E\bigg[\Phi^{p/2}_{p,L^\sj}\big(f^{(2)}_t(0,\cdot)\big)\bigg],\]
which is finite by Lemma \ref{PhiQGro} and assumption (B1).
\end{proof}

With these preliminaries at hand, we can finally prove Theorem \ref{theorem-sigma} (iii). As we remarked at the beginning of Subsection \ref{stableproof}, it is sufficient to show convergence in probability for a fixed $t>0$ in order to obtain uniform convergence on compacts in probability. Therefore, the theorem is an immediate consequence of the following lemma and Lemma 4.3 in \cite{BLP}.

\begin{lem} 
Under the conditions of Theorem \ref{theorem-sigma} (iii), there is a process $(Z_t)_{t\geq 0}$ that satisfies almost surely $V(Z,p;k)_t^n=V(X,p;k)_t^n$ for all $n\in\N$ and $t\geq 0$, and has the following properties. The sample paths of $Z$ are with probability 1 $k$-times differentiable and it holds for Lebesgue almost all $t\geq 0$ that
\[\frac{\partial^k Z_t}{(\partial t)^k}=\int_{-\infty}^t g^{(k)} (t-s)\sigma_{s-}\diff L_s.\]
Moreover, for any $t_0>0$ it holds with probability 1 that
\[\frac{\partial^k Z_t}{(\partial t)^k}\in L^p([0,t_0]).\]
\end{lem}

\begin{proof}
For ease of notation we consider the case $k=1.$ The general case follows by similar arguments.
The strategy of the proof is the following. We let $a\in(0,1]$ and define the processes $(F^\sj_u)_{u\in\R}$ and $(F^\lj_u)_{u\in\R}$ by
\begin{align*}
F^\sj_u&=\int_{-\infty}^u g'(u-s)\sigma_{s-}\diff L^\sj_s,\quad\text{and}\quad
F^\lj_u&=\sum_{s\in(-\infty, u)}g'(u-s)\sigma_{s-} \Delta L_s \mathds 1_{\{|\Delta L_s|>a\}},
\end{align*}
where the process $(L^\sj_t)_{t \in \mathbb R}$ has been introduced in Section \ref{SmaJum}.
We show that both processes $F^\sj_u$ and $F^\lj_u$ are well-defined and that they both admit a modification with sample paths in $L^p([0,t]).$ Then, we define the process
\[Z_t:=\int_0^t (F_u^\sj+F_u^\lj)\diff u,\]
and show that it satisfies the properties given in the lemma. 

We begin by analysing $F^\sj_u.$ The process $f_{t_0}(u,s)=g'(u-s)\sigma_{s-}\mathds 1_{[0, t_0]}(u)\mathds 1_{(-\infty,u)}(s)$ is integrable in $s$ with respect to $L^\sj$ for Lebesgue almost all $u$, according to  Lemma \ref{SmaJumLpEst}.  
As we argued in Lemma \ref{Fub}, $(F_u^\sj)_{u\geq 0}$ admits a modification with measurable sample paths. Moreover, applying Lemmas \ref{KwaWoyIso} and \ref{SmaJumLpEst} we obtain $F^\sj\in L^p([0,t])$, almost surely, since
\[\E\bigg[\int_{0}^{t}|F^\sj_u|^p\diff u\bigg]\leq C\int_0^{t}\E[\|f_{t}(u,\cdot)\|^p_{p,L^\sj}]\diff u<\infty.\]
For the process $F^\lj_u$ we make the decomposition 
\begin{align*}
F^{\lj}_u	&=F^{\slj}_u+F^{\llj}_u\\
			&=\sum_{s\in(-\infty, -\delta]}g'(u-s)\sigma_{s-} \Delta L_s \mathds 1_{\{|\Delta L_s|>a\}}+\sum_{s\in(-\delta, 
			u)}g'(u-s)\sigma_{s-} \Delta L_s \mathds 1_{\{|\Delta L_s|>a\}}.
\end{align*}
We argue first that $F^{\slj}$ is well-defined and in $L^p([0,t])$ almost surely. Applying Lemma~\ref{poiintbou} we obtain that
\begin{align}\label{poiintest}
\sum_{s\in(-\infty, -\delta]}|g'(-s)\sigma_{s-} \Delta L_s| \mathds 1_{\{|\Delta L_s|>a\}}<\infty
\end{align}
almost surely. Since $g'$ is decreasing on $[\delta,\infty),$ we conclude that the sum in the definition of $F_u^{\slj}$ is absolutely convergent and bounded by the left hand side of \eqref{poiintest}, which does not depend on $u.$ It follows that $F^{\slj}$ is well-defined and in $L^p([0,t])$ almost surely.

 For $F^{\llj}_u$ we use that $L$ has only finitely many jumps of size $>a$ on $[-\delta,t]$. Therefore, $F^{\llj}$ is well-defined and we find a positive random variable $K<\infty$ such that
\begin{align*}
\int_0^{t}|F_u^{\llj}|^p\diff u
	&=\int_0^{t} \bigg|\sum_{s\in(-\delta, u)}g'(u-s)\sigma_{s-} \Delta L_s \mathds 1_{\{|\Delta L_s|>a\}}\bigg|^p\diff u\\
	&\leq K\int_0^{t}\sum_{s\in(-\delta, u)}\big|g'(u-s)\sigma_{s-} \Delta L_s \mathds 1_{\{|\Delta L_s|>a\}}\big|^p\diff u\\
	&\leq K\sum_{s\in(-\delta, {t})}\big|\sigma_{s-} \Delta L_s \mathds 1_{\{|\Delta L_s|>a\}}\big|^p\int_0^{t}|g'(u-s)|^p\diff u\\
	&<\infty.
\end{align*}
Here, the last inequality follows by $|g'(s)|\leq Cs^{\alpha-1}$ for $s\in (0, \delta)$ and $(\alpha-1)p>-1.$

Define the process $(Z_t)_{t\geq 0}$ by
\[Z_t:=\int_0^t (F^\sj_u+F^\lj_u)\diff u.\]
All that remains to show is that $V(X,p;1)^n_{t}=V(Z,p;1)^n_{t}$ for all $n\in \N$ and all $t>0$ with probability 1. For any $t>0$ it holds with probability 1 that
\begin{align*}
X_t-X_0=\int_\R \bigg(\int_\R f_t(u,s)\diff u\bigg) \diff L_s= \int_\R \bigg(\int_\R f_t(u,s)\diff L_s\bigg) \diff u=Z_t,
\end{align*}
where we have applied Lemmas \ref{Fub} and \ref{FubCon}. Consequently, it holds that $\P[X_t=Z_t+X_0\ $for all $t\in\Q_+ ]=1$ which implies $V(X,p;1)^n_{t}=V(Z,p;1)^n_{t}$ for all $n\in\N$ and all $t>0$ almost surely.
\end{proof}

\bibliographystyle{chicago}
 
\end{document}